\documentclass[a4paper, 11pt, twoside, notitlepage]{amsart}

\usepackage{amsmath, amssymb, amsthm, amscd}
\usepackage{mathtools}
\usepackage{mathrsfs}

\usepackage{graphicx, xcolor}
\usepackage{tikz}

\usepackage[ocgcolorlinks, linkcolor=blue]{hyperref}
\usepackage{url}
\usepackage[shortlabels]{enumitem}
\usepackage{comment}
\usepackage{verbatim}

\usepackage{bm}
\usepackage{bbm}
\usepackage{dsfont}
\usepackage{relsize}
\usepackage{esint}
\usepackage[normalem]{ulem}

\usepackage[utf8]{inputenc}

\newcommand\redsout{\bgroup\markoverwith{\textcolor{red}{\rule[0.5ex]{2pt}{0.4pt}}}\ULon}

\numberwithin{equation}{section}

\renewcommand{\H}{{\mathbb H}}

\allowdisplaybreaks

\mathtoolsset{showonlyrefs}

\graphicspath{{images/}}

\newtheorem{theorem}{Theorem}[section]
\newtheorem{lemma}[theorem]{Lemma}

\newtheorem{proposition}[theorem]{Proposition}
\newtheorem{question}{Question}

\newtheorem{remark}[theorem]{Remark}

\title[Entanglement principle on hyperbolic spaces]{Entanglement principle for fractional Laplacian on hyperbolic spaces and applications to inverse problems}

\author[Y.-H. Lin]{Yi-Hsuan Lin}
\address{Department of Applied Mathematics, National Yang Ming Chiao Tung University, Hsinchu, Taiwan \& Fakult\"at f\"ur Mathematik, University of Duisburg-Essen, Essen, Germany}
\curraddr{}
\email{yihsuanlin3@gmail.com}



\newcommand{\C}{{\mathbb C}}
\newcommand{\R}{{\mathbb R}}
\newcommand{\Z}{{\mathbb Z}}
\newcommand{\N}{{\mathbb N}}

\newcommand{\Id}{\mathrm{Id}}

\newcommand{\LC}{\left(}
\newcommand{\RC}{\right)}
\newcommand{\wt}{\widetilde}

\newcommand{\norm}[1]{\lVert #1 \rVert}
\DeclareMathOperator{\supp}{supp} 

\begin{document}

	\maketitle
	\begin{abstract}
		We establish an entanglement principle for fractional powers of the Laplace--Beltrami operator on hyperbolic space \(\H^n\), \(n\ge 2\). More precisely, we prove that if finitely many distinct noninteger powers of \(-\Delta_{\H^n}\), acting on functions that vanish on a common nonempty open set, satisfy a linear dependence relation on that set, then each of these functions must vanish identically on \(\H^n\). This extends the recently developed entanglement principle for the fractional Laplacian on \(\R^n\) to the negatively curved setting of hyperbolic space. As an application, we derive global uniqueness results for inverse problems associated with fractional polyharmonic equations on \(\H^n\), including a fractional Calder\'on problem. The proof relies on the heat semigroup representation of fractional powers together with sharp global heat kernel estimates on hyperbolic space.
	\end{abstract}

	\tableofcontents

	\section{Introduction}\label{sec:introduction}
	
	Inverse problems for fractional differential equations have been investigated extensively in recent years in a variety of geometric and analytic settings. 
	A foundational result in this direction is the fractional Calder\'on problem for the fractional Schr\"odinger equation, first solved in \cite{GSU20}. 
	Due to the intrinsic nonlocality of the operator, the problem is naturally formulated as an exterior value problem. 
	More precisely, given $s\in (0,1)$, let $\Omega \subset\R^n$ be a bounded open set with Lipschitz boundary for $n\in \N$, and let $q\in L^\infty(\Omega)$. 
	Consider the exterior value problem 
	\begin{equation}\label{eq: fractional Schro}
		\begin{cases}
			\LC (-\Delta)^s +q \RC u=0 &\text{ in }\Omega, \\
			u=f &\text{ in }\R^n\setminus \overline{\Omega}.
		\end{cases}
	\end{equation}
	When $0$ is not a Dirichlet eigenvalue of $(-\Delta)^s+q$ in $\Omega$, there exists a unique solution $u\in H^s(\R^n)$ to \eqref{eq: fractional Schro} for any given exterior Dirichlet data $f$, for instance $f\in C^\infty_c(\Omega_e)$ or $f\in H^s(\Omega_e)$. 
	In this case, one can define the exterior (partial) \emph{Dirichlet-to-Neumann} (DN) map by
	\begin{equation}\label{nonlocal DN}
		\Lambda_q^{s} : f\mapsto  \left. (-\Delta)^s u_f \right|_{W_2},
	\end{equation}
	where $W_1, W_2\Subset \Omega_e$ are any nonempty open subsets.
	
	In the seminal work \cite{GSU20}, it was shown that the DN map \eqref{nonlocal DN} uniquely determines the bounded potential $q$ in $\Omega$ for all dimensions $n\in \N$. 
	This result highlights a remarkable rigidity phenomenon: in contrast to the classical local case $s=1$, where global uniqueness may depend sensitively on the dimension ($n=2$ and $n\geq 3$), and additional assumptions, the nonlocal operator exhibits stronger global determination properties.
	
	Two fundamental ingredients underlying these results are the \emph{unique continuation property} (UCP) and the \emph{Runge approximation property}. 
	Loosely speaking, given a nonempty open subset $\mathcal{O}\subset \R^n$, the UCP for $(-\Delta)^s$ ($0<s<1$) asserts that
	\[
	u=(-\Delta)^su =0 \text{ in }\mathcal{O}
	\implies
	u\equiv 0 \text{ in } \R^n.
	\]
	Moreover, the Runge approximation property states that any $L^2$ function can be approximated by solutions of the fractional Schr\"odinger equation. 
	These two features are substantially stronger than in the local case and play a decisive role in fractional inverse problems.
	
	Based on these properties, fractional inverse problems have attracted rapidly growing attention in the past decade. 
	In \cite{cekic2020calderon}, it was shown that both first-order and zeroth-order coefficients can be uniquely determined from the DN map, a phenomenon that generally fails in the local case without additional structural assumptions. 
	A simultaneous recovery of an obstacle and a surrounding coefficient was investigated in \cite{CLL2017simultaneously}, and the determination of bounded potentials for anisotropic fractional Schr\"odinger equations was established in \cite{GLX}. 
	Notably, these problems remain open in the classical case $s=1$ for $n\geq 3$, further illustrating the structural advantages induced by nonlocality.
	
	Beyond uniqueness, quantitative aspects have also been studied. 
	Monotonicity-based if-and-only-if methods were developed for both linear and semilinear fractional equations in \cite{harrach2017nonlocal-monotonicity,harrach2020monotonicity, lin2020monotonicity}. 
	Global uniqueness from a single measurement and measurable UCP were proved in \cite{GRSU20}. 
	Optimal and logarithmic stability results were obtained in \cite{ruland2018exponential, RS20, KLW2021calder}. 
	We also refer the reader to
	\cite{LL2022inverse, LL2020inverse,CMRU20,LLR2019calder,LZ2023unique,GU2021calder,LLU2022para} for further developments and references.
	All these works focus on recovering lower-order perturbations of nonlocal operators and demonstrate that nonlocality enhances both rigidity and approximation properties compared to the local setting.
	
	More recently, interior determination (leading coefficient determination) problems have been studied using alternative techniques. 
	One approach relies on the Caffarelli-Silvestre (CS) extension (see, e.g., \cite{CGRU2023reduction,ruland2023revisiting, LLU2023calder, LZ2024approximation}), which relates the nonlocal and local equations via the Caffarelli-Silvestre extension.
	Another approach is based on the heat semigroup representation of fractional operators (see, e.g., \cite{feizmohammadi2021fractional, Fei24_TAMS, FKU24, lin2024fractional}). 
	The heat semigroup method applies uniformly to general fractional powers and to operators on Riemannian manifolds. 
	In particular, while the CS extension is well understood for single fractional powers, no comparable extension is available for general fractional polyharmonic operators. 
	In contrast, the heat semigroup formulation provides a flexible framework for handling mixed fractional powers.
	
	A particularly intriguing development in this direction is the introduction of an \emph{entanglement principle} for nonlocal operators. 
	Roughly speaking, this principle states that if several distinct fractional powers of an operator, acting on functions that vanish on a common nonempty open set, are linearly dependent on that set, then each function must vanish identically. 
	Such a phenomenon was first established for nonlocal elliptic operators on closed Riemannian manifolds in \cite{FKU24} (compact case), and later for the fractional Laplacian on $\R^n$ in \cite{FL24} (noncompact case). 
	More recently, \cite{LLY25_entangle} extended the entanglement principle to nonlocal parabolic operators on a noncompact Euclidean domain. 
	This mechanism provides a powerful tool for decoupling mixed fractional powers and has already led to new results in inverse problems. 
	Finally, the recent monograph \cite{LL25_Integro} investigates comprehensive studies for spatial fractional inverse problems.

	\subsection{Mathematical models}
	
	In this paper, we investigate inverse problems associated with the fractional Laplacian on hyperbolic spaces. 
	Several equivalent models of the $n$-dimensional hyperbolic space $\H^n$ appear in the literature. 
	Throughout this work, we adopt the hyperboloid model.
	
	More precisely, let $\R^{n+1}$ be equipped with the Lorentzian metric 
	\[
	-dx_0^2 + dx_1^2 + \cdots + dx_n^2,
	\]
	and denote by $[\cdot,\cdot]$ the associated Lorentzian inner product
	\begin{equation}\label{eq: Lorentzian metric}
		[x,y] = x_0y_0 - x_1y_1 - \cdots - x_ny_n,
	\end{equation}
	for $x=(x_0,\ldots,x_n)$ and $y=(y_0,\ldots,y_n)\in \R^{n+1}$. 
	
	The real hyperbolic space $\H^n$ is defined as the upper sheet of the two-sheeted hyperboloid
	\begin{equation}
		\begin{split}
			\H^n 
			&= \left\{ (x_0,\ldots,x_n)\in \R^{n+1} : \, x_0^2 - x_1^2 - \cdots - x_n^2 = 1, \ x_0>0 \right\} \\
			&= \left\{ (\cosh r, \sinh r\, \omega) : \, r \ge 0, \ \omega \in \mathbb S^{n-1} \right\}.
		\end{split}
	\end{equation}
	Equipped with the induced Riemannian metric, $\H^n$ is a complete, simply connected manifold of constant sectional curvature $-1$, and it is the simplest example of a rank-one symmetric space.
	
	In geodesic polar coordinates $(r,\omega)$ centered at the point $e_0=(1,0,\ldots,0)$, the metric takes the form
	\begin{equation}\label{eq: metric of hyperbolic}
		g_{\H^n} = dr^2 + \sinh^2 r \, d\omega^2,
	\end{equation}
	where $d\omega^2$ denotes the standard metric on the sphere $\mathbb S^{n-1}$. 
	The associated Riemannian distance between $x,y\in \H^n$ will be denoted by $d_{\H^n}(x,y)$.
	
	With respect to these coordinates, the Laplace-Beltrami operator is given by
	\begin{equation}
		\Delta_{\H^n}= \partial_{rr}+ (n-1)\frac{\cosh r}{\sinh r}\partial_r + \frac{1}{\sinh^2 r}\Delta_{\mathbb S^{n-1}},
	\end{equation}
	and the corresponding volume element reads
	\begin{equation}\label{eq: volume in polar}
		dV(x) = \sinh^{n-1} r \, dr \, d\omega,
	\end{equation}
	where $d\omega$ denotes the surface measure on $\mathbb S^{n-1}$.

	\subsection{Entanglement principle for the fractional Laplace operator}
	
	We investigate whether a unique continuation phenomenon holds for mixed fractional Laplace operators on $\H^n$. 
	More precisely, let $N\in \N$ and let $\mathcal O\subset \H^n$ be a nonempty open set. 
	Suppose that $u\in H^{r}(\H^n)$ for some $r\in \R$ and that
	\begin{equation}\label{eq: UCP_poly}
		u|_{\mathcal O}= \sum_{k=1}^N b_k \big( (-\Delta_{\H^n})^{s_k} u \big)\Big|_{\mathcal O} =0,
	\end{equation}
	for coefficients $\{b_k\}_{k=1}^N \subset \C\setminus \{0\}$ and exponents 
	$\{s_k\}\subset (0,\infty)\setminus \N$. 
	Does it follow that $u\equiv 0$ in $\H^n$?
	
	To the best of our knowledge, results at this level of generality are not available even in the case $N=1$ on hyperbolic space.
	In the Euclidean setting, unique continuation for a single fractional Laplacian is often proved using the CS extension \cite{CS07}, which realizes $(-\Delta)^s$ as a Dirichlet-to-Neumann relation for a degenerate elliptic equation. 
	This approach has been instrumental in establishing UCP results for fractional operators on $\R^n$ (see, e.g., \cite{ruland2015unique, GSU20}). 
	However, no comparable extension is known for general fractional polyharmonic operators, neither on $\R^n$ nor on $\H^n$. We also refer the readers to \cite{BGS_15} for the CS extension on hyperbolic spaces.
	
	Instead, we adopt the heat semigroup characterization of fractional powers, which provides a flexible framework on general Riemannian manifolds and is well-suited to mixed fractional orders.
	This viewpoint will play a central role in our analysis.
	
	We will show that \eqref{eq: UCP_poly} is a special case of a more general rigidity phenomenon, which we refer to as the \emph{entanglement principle} for fractional Laplace operators on $\H^n$.
	
	\begin{question}\label{question}
		Let $N\in \N$, let $\{s_k\}_{k=1}^N\subset (0,\infty)\setminus \N$, and let $\mathcal{O}\subset \H^n$ be a nonempty open set. 
		Suppose $\{u_k\}_{k=1}^N\subset H^r(\H^n)$\footnote{The fractional Sobolev space on $\H^n$ will be introduced in Section \ref{sec:prel}.} 
		for some $r\in \R$, and assume that 
		\begin{equation}\label{ent_u_cond}
			u_1|_{\mathcal O}=\ldots=u_N|_{\mathcal O}=0 
			\quad \text{and} \quad  
			\sum_{k=1}^N  b_k(-\Delta_{\H^n})^{s_k}u_k\Big|_{\mathcal O}=0,
		\end{equation}
		for some $\{b_k\}_{k=1}^N\subset \C\setminus \{0\}$. 
		Does it follow that $u_k\equiv 0$ in $\H^n$ for all $k=1,\ldots,N$?
	\end{question}

	When $N=1$, this reduces to the classical unique continuation problem for the fractional Laplacian on $\H^n$, which is not yet understood in full generality.
	For $N\ge 2$, the statement is substantially stronger, as it involves several functions whose fractional powers are linearly dependent on an open set. 
	It is important to note that Question~\ref{question} cannot hold without additional restrictions on the exponents. 
	Indeed, consider $N=2$, let $s_1=s\in (0,1)$ and $s_2=s_1+m$ for some $m\in \N$, and choose a nontrivial function $u\in C^\infty_c(\H^n\setminus \overline{\mathcal{O}})$. 
	Setting $u_1 = (-\Delta_{\H^n})^m u$ and $u_2=u$, we obtain
	\[
	(-\Delta_{\H^n})^{s_1}u_1 -(-\Delta_{\H^n})^{s_2}u_2=0
	\quad \text{in } \H^n,
	\]
	which shows that resonance between exponents may cause the entanglement principle to fail.
	
	To exclude such resonance phenomena, we impose the following nondegeneracy condition.
	
	\begin{enumerate}[\textbf{(H)}]
		\item\label{exponent condition}
		The exponents satisfy $s_1<s_2<\ldots <s_N$ and 
		$s_k-s_j \notin \Z$ for all $j\neq k$.
	\end{enumerate}
	
	This assumption ensures that the fractional powers remain genuinely distinct and cannot be reduced to one another by integer shifts.
	
	We can now state the main entanglement result.
	
	\begin{theorem}[Entanglement principle]\label{thm: ent}
		Let \(O\subset \H^n\) be a nonempty open set with \(n\ge 2\). Let
		\(N\in \N\) and \(\{s_k\}_{k=1}^N\subset (0,\infty)\setminus \N\) satisfy Assumption \((H)\).
		Assume that \(\{u_k\}_{k=1}^N \subset H^{-r}(\H^n)\) for some \(r\ge 0\), and let
		\[
		\rho_\gamma(x):=e^{-\gamma\sqrt{1+d_{\H^n}(e_0,x)^2}}, \qquad x\in \H^n,
		\]
		for some \(\gamma>\frac{n-1}{2}\).
		Suppose that for each \(k=1,\dots,N\),
		\begin{equation}\label{eq: exponential decay condition weakly}
			|\langle u_k,\phi\rangle|
			\le C\,\|\rho_\gamma \phi\|_{H^r(\H^n)}
			\qquad \text{for all }\phi\in C_c^\infty(\H^n).
		\end{equation}
		If
		\[
		u_1|_O=\cdots=u_N|_O=0
		\qquad\text{and}\qquad
		\left.\left(\sum_{k=1}^N b_k(-\Delta_{\H^n})^{s_k}u_k\right)\right|_O=0,
		\]
		for some \(\{b_k\}_{k=1}^N\subset \C\setminus\{0\}\), then
		\[
		u_k\equiv 0 \qquad \text{in }\H^n,\quad k=1,\dots,N.
		\]
	\end{theorem}
	
	To the best of our knowledge, this is the first entanglement principle established on a noncompact negatively curved manifold, extending previous Euclidean and compact manifold results to the hyperbolic setting.
	We note that the decay condition \eqref{eq: exponential decay condition weakly} is automatically satisfied for functions supported in a bounded subset of $\H^n$ (see Lemma~\ref{lem:compact_support_decay}).		
	The exponential decay assumption \eqref{eq: exponential decay condition weakly} guarantees sufficient spatial decay to justify repeated integration by parts in the heat semigroup representation (see Section \ref{sec:entangle}). 
	Here $H^r(\H^n)$ denotes the fractional Sobolev space introduced in Section~\ref{sec:prel}, and $\langle\cdot,\cdot\rangle$ is the duality pairing between fractional Sobolev spaces $H^{-r}$ and $H^{r}$.
	
	\subsection{An application: the fractional Calder\'on problem}
	
	We next apply the entanglement principle to an inverse problem for fractional polyharmonic equations on $\H^n$. 
	Let $\{s_k\}_{k=1}^N \subset (0,\infty)\setminus \N$ satisfy Assumption~\ref{exponent condition}, 
	let $\Omega\subset \H^n$ be a bounded open set with Lipschitz boundary $\partial\Omega$, 
	and let $q\in L^\infty(\Omega)$. 
	Define the fractional polyharmonic operator
	\begin{equation}\label{eq: fractional poly op.}
		\mathcal{P}_{\H^n}:= \sum_{k=1}^N b_k (-\Delta_{\H^n})^{s_k},
	\end{equation}
	where $b_k>0$ for $k=1,\ldots, N$. The positivity of $\{b_k\}_{k=1}^N$ is used to guarantee the well-posedness of the following exterior value problem
	\begin{equation}\label{eq: main}
		\begin{cases}
			\big( \mathcal{P}_{\H^n}+q \big) u=0 &\text{ in }\Omega , \\
			u=f&\text{ in }\Omega_e,
		\end{cases}
	\end{equation}
	where $\Omega_e :=\H^n \setminus \overline{\Omega}$.
	To ensure well-posedness, we assume
	\begin{equation}\label{eq: eigenvalue condition}
		0 \text{ is not a Dirichlet eigenvalue of \eqref{eq: main}.}
	\end{equation}
	When $N=1$, this reduces to the standard fractional Schr\"odinger equation on $\Omega\subset \H^n$.
	
	Under this condition, one can define the partial DN map
	\[
	\Lambda_q : \wt H^{s_N}(W_1)\to H^{-s_N}(W_2),
	\qquad 
	f\mapsto \mathcal{P}_{\H^n} u_f \big|_{W_2},
	\]
	where $u_f$ solves \eqref{eq: main} and $W_1,W_2\Subset \Omega_e$ are nonempty open sets.
	Then we can prove the following global uniqueness result.
	
	\begin{theorem}\label{thm: main}
		Let $\Omega\subset \H^n$ be a bounded open set with Lipschitz boundary and $n\ge 2$. 
		Let $q_j\in L^\infty(\Omega)$ for $j=1,2$, and assume \eqref{eq: eigenvalue condition} holds for each $q_j$. 
		If the corresponding DN maps satisfy
		\begin{equation}\label{eq: same DN map}
			\Lambda_{q_1}f \big|_{W_2}=\Lambda_{q_2}f  \big|_{W_2}
			\quad \text{for all }f\in C^\infty_c(W_1),
		\end{equation}
		then $q_1=q_2$ in $\Omega$.
	\end{theorem}
	
	The key new ingredient, compared to the Euclidean case, is the entanglement principle on $\H^n$, which replaces the role of classical unique continuation in decoupling mixed fractional powers. A similar result in $\R^n$ has been studied in \cite{FL24}.
	
	\begin{remark}
		As a corollary of Theorem~\ref{thm: main}, if $s\in (0,\infty)\setminus \N$, then the potential in the fractional Schr\"odinger equation
		\[
		\begin{cases}
			\big( (-\Delta_{\H^n})^s +q \big) u =0 &\text{ in }\Omega, \\
			u=f &\text{ in }\Omega_e,
		\end{cases}
		\]
		is uniquely determined by its partial DN map.
	\end{remark}
	
	\noindent\textbf{Organization of the paper.}
	In Section~\ref{sec:prel}, we introduce the fractional Laplacian on $\H^n$ via the Helgason--Fourier transform and the heat semigroup representation, and recall the associated fractional Sobolev spaces. 
	We also establish well-posedness of the exterior Dirichlet problem and define the DN map. 
	Section~\ref{sec:entangle} is devoted to the proof of the entanglement principle, which relies on quantitative heat kernel estimates and a decoupling argument for mixed fractional powers. 
	Finally, in Section~\ref{sec:IP}, we combine the entanglement principle with a Runge approximation argument to prove the global uniqueness result for the fractional Calder\'on problem.

	\section{Preliminaries}\label{sec:prel}
	
	In this section, we collect several basic definitions and properties that will be used throughout the paper.
	
	\subsubsection*{Notation}
	Throughout the paper, given nonnegative quantities $A$ and $B$, we write
	$A \lesssim B$ if there exists a constant $C>0$ (independent of the relevant variables and functions under consideration) such that
	$A \le C\,B$.
	We write $A \gtrsim B$ if $B \lesssim A$, and $A \simeq B$ if both $A \lesssim B$ and $A \gtrsim B$ hold.
	We also write $A \sim B$ to indicate two-sided comparability in the same sense.
	The implicit constants may change from line to line.
	
	\subsection{The fractional Laplacian on hyperbolic spaces}
	
	We briefly review several characterizations of the fractional Laplacian on $\H^n$.
	
	\subsubsection{The Helgason--Fourier transform}
	
	Following \cite{BGS_15}, for $s\in (0,1)$ the fractional Laplacian $(-\Delta_{\H^n})^{s}$ on $\H^n$ can be defined via the Helgason--Fourier transform. 
	We begin by recalling the Euclidean definition: for $u\in \mathcal{S}(\R^n)$,
	\begin{equation}\label{eq: def of fractional Laplacian via Fourier}
		(-\Delta)^s u = \mathcal{F}^{-1}\big( |\xi|^{2s}\widehat{u}\big),
	\end{equation}
	where
	\begin{equation}\label{eq: Fourier transform}
		\widehat{u}(\xi)=\mathcal{F}(u)(\xi):=\int_{\R^n}u(x)e^{-\mathsf{i}x\cdot \xi}\, dx,
	\end{equation}
	and $\mathcal{F}^{-1}$ denotes the inverse Fourier transform ($\mathsf{i}=\sqrt{-1}$). 
	Note that the plane waves $e^{-\mathsf{i}x\cdot \xi}$ are generalized eigenfunctions of $\Delta$ with eigenvalue $-|\xi|^2$.
	
	Similarly, on $\H^n$ one considers generalized eigenfunctions of the Laplace--Beltrami operator,
	\begin{equation}
		h_{\lambda,\theta}(x)=[x,(1,\theta)]^{\mathsf{i}\lambda-(n-1)/2}, 
		\qquad x\in \H^n,
	\end{equation}
	where $[\cdot,\cdot]$ is the Lorentzian inner product in \eqref{eq: Lorentzian metric}, $\lambda\in \R$, and $\theta\in \mathbb{S}^{n-1}$. 
	These satisfy
	\begin{equation}
		\Delta_{\H^n}h_{\lambda,\theta}
		= -\Big(\lambda^2+\frac{(n-1)^2}{4}\Big)h_{\lambda,\theta}.
	\end{equation}
	The Helgason--Fourier transform (see \cite[Chapter III, equation (4)]{Hel_book_08}) is defined by
	\begin{equation}
		\mathcal{F}_{\mathcal{H}}(f)(\lambda,\theta)
		:= \int_{\H^n} f(x)\, h_{\lambda,\theta}(x)\, dV(x),
	\end{equation}
	with the inversion formula
	\begin{equation}
		f(x)=\int_{-\infty}^{\infty}\int_{\mathbb{S}^{n-1}}
		\overline{h_{\lambda,\theta}(x)}\,\mathcal{F}_{\mathcal{H}}(f)(\lambda,\theta)\,
		\frac{d\theta\, d\lambda}{|c(\lambda)|^2},
	\end{equation}
	where $c(\lambda)$ is Harish--Chandra's $c$-function given by
	\begin{equation}
		\frac{1}{|c(\lambda)|^2}
		= \frac{1}{2(2\pi)^n}\,
		\frac{\big|\Gamma\big(\mathsf{i}\lambda +\frac{n-1}{2}\big)\big|^2}{|\Gamma(\mathsf{i}\lambda)|^2}.
	\end{equation}
	Moreover, one has the Plancherel identity
	\begin{equation}\label{eq: Plancherel formula}
		\norm{f}_{L^2(\H^n)}^2
		=\int_{\H^n}|f(x)|^2\, dV(x)
		=\int_{\R \times \mathbb{S}^{n-1}}
		\big|\mathcal{F}_{\mathcal{H}}(f)(\lambda,\theta)\big|^2\,
		\frac{d\theta\, d\lambda}{|c(\lambda)|^2}.
	\end{equation}
	In addition, for $f\in L^2(\H^n)$,
	\begin{equation}
		\begin{split}
			\mathcal{F}_{\mathcal{H}}(\Delta_{\H^n}f)(\lambda,\theta)
			&=\int_{\H^n} \Delta_{\H^n} f(x)\, h_{\lambda,\theta}(x)\, dV(x)
			=\int_{\H^n} f(x)\, \Delta_{\H^n} h_{\lambda,\theta}(x)\, dV(x)\\
			&= -\Big(\lambda^2+\frac{(n-1)^2}{4}\Big)\mathcal{F}_{\mathcal{H}}(f)(\lambda,\theta).
		\end{split}
	\end{equation}
	Consequently, analogously to \eqref{eq: def of fractional Laplacian via Fourier}, we define
	\begin{equation}\label{eq: def of fractional Laplacian on H^n}
		\mathcal{F}_{\mathcal{H}}\big((-\Delta_{\H^n})^{s} f\big)(\lambda,\theta)
		=\Big(\lambda^2+\frac{(n-1)^2}{4}\Big)^{s}\mathcal{F}_{\mathcal{H}}(f)(\lambda,\theta).
	\end{equation}
	
	\subsubsection{The singular integral}
	
	The operator $(-\Delta_{\H^n})^{s}$ also admits a singular integral representation, established in \cite{BGS_15}.
	
	\begin{proposition}[\text{\cite[Theorem 2.5]{BGS_15}}]
		Let $s\in (0,1)$ and $n\ge 2$. Then
		\begin{equation}
			(-\Delta_{\H^n})^s f(x)
			= c_{n,s}\,\mathrm{P.V.}\int_{\H^n}\big(f(x)-f(y)\big)\,
			\mathcal{K}_{n,s}\big(d_{\H^n}(x,y)\big)\, dV(y),
		\end{equation}
		where
		\begin{equation}
			\begin{split}
				\mathcal{K}_{n,s}(\rho)
				:= \begin{cases}
					C_1 \Big( -\frac{\partial_{\rho}}{\sinh \rho} \Big)^{\frac{n-1}{2}}
					\Big( \rho^{-\frac{1+2s}{2}}K_{\frac{1+2s}{2}}\big( \frac{n-1}{2}\rho \big) \Big),
					&\text{for odd }n\ge 3,\\[0.3em]
					C_1 \displaystyle\int_{\rho}^{\infty}\frac{\sinh r}{\sqrt{\pi}\sqrt{\cosh r -\cosh \rho}}
					\Big( -\frac{\partial_{r}}{\sinh r} \Big)
					\Big( r^{-\frac{1+2s}{2}}K_{\frac{1+2s}{2}}\big( \frac{n-1}{2}r\big) \Big)\, dr,
					&\text{for even }n\ge 2,
				\end{cases}
			\end{split}
		\end{equation}
		$\mathrm{P.V.}$ denotes the Cauchy principal value, $K_\nu$ is the modified Bessel function of the second kind, and the constants are
		\begin{equation}\label{eq: constant C_1 and c_n,s}
			c_{n,s}:= \frac{8\sqrt{2}\Gamma(n+s)}{3\Gamma\big( \frac{n}{2}\big)\Gamma(-s)},
			\qquad
			C_1 := \frac{1}{2^{n-2+2s}\Gamma\big( \frac{n-1}{2}\big)\Gamma \big( \frac{1+2s}{2}\big)}.
		\end{equation}
	\end{proposition}
	
	In \cite[Section 2]{BGS_15}, it is shown that $\mathcal{K}_{n,s}(\rho)>0$ and that
	\begin{equation}
		\begin{split}
			\mathcal{K}_{n,s}(\rho) &\sim \rho^{-n-2s}  \qquad\qquad \text{as }\rho \to 0^+,\\
			\mathcal{K}_{n,s}(\rho) &\sim \rho^{-1-s}e^{-(n-1)\rho} \qquad \text{as }\rho \to \infty,
		\end{split}
	\end{equation}
	in the sense of the comparability notation introduced above.
	We refer to \cite{BGS_15, Hel_book_08, GGG_book_03} for further discussion of Fourier-analytic constructions of $(-\Delta_{\H^n})^s$.
	
	\subsubsection{The heat semigroup}
	
	We will also use the heat semigroup formulation of fractional powers. 
	Let $p_t(x,y)$ be the heat kernel on $\H^n$ for $t>0$, characterized by:
	\begin{enumerate}[(i)]
		\item $(\partial_t-\Delta_{\H^n})p_t(\cdot,y)=0$ for $t>0$;
		\item $p_t(\cdot,y)\to \delta_y$ as $t\to 0^+$ in the sense of distributions;
		\item (semigroup property) $p_{t+t'}(x,y)=\int_{\H^n} p_t(x,z)p_{t'}(z,y)\, dV(z)$;
		\item (symmetry) $p_t(x,y)=p_t(y,x)$.
	\end{enumerate}
	Moreover, $p_t(x,y)=p_t(\rho)$ with $\rho=d_{\H^n}(x,y)$ admits an explicit representation. 
	For our purposes, we mainly use the global heat kernel bounds (see \cite[Section 3]{DM88_heatkernel_hyper}):
	\begin{equation}\label{eq: global bounds heat kernel}
		p_t(\rho) \sim (1+\rho)(1+\rho+t)^{(n-3)/2}t^{-n/2}
		\exp\Big(-\frac{(n-1)^2}{4}t-\frac{n-1}{2}\rho -\frac{\rho^2}{4t}\Big),
		\qquad \rho \ge 0,\ t>0.
	\end{equation}
	These bounds will play a key role in the proof of the entanglement principle in Section~\ref{sec:entangle}.
	
	Using functional calculus, one may define the fractional Laplacian via the heat semigroup by
	\begin{equation}\label{eq: fractional via heat}
		(-\Delta_{\H^n})^s u(x)
		:= \frac{1}{\Gamma(-s)}\int_0^\infty\big(e^{t\Delta_{\H^n}}u(x)-u(x)\big)\,\frac{dt}{t^{1+s}},
		\qquad x\in \H^n,
	\end{equation}
	for $u\in \mathrm{Dom}\big((-\Delta_{\H^n})^s\big)
	=\{u\in L^2(\H^n):\, (-\Delta_{\H^n})^s u\in L^2(\H^n)\}$, where
	\[
	e^{t\Delta_{\H^n}}u(x):=\int_{\H^n} p_t(x,y)u(y)\, dV(y)
	\]
	solves
	\[
	\begin{cases}
		(\partial_t-\Delta_{\H^n})\big(e^{t\Delta_{\H^n}}u\big)=0 &\text{in }\H^n\times (0,\infty),\\
		e^{t\Delta_{\H^n}}u\big|_{t=0}=u &\text{in }\H^n.
	\end{cases}
	\]
	The definition \eqref{eq: fractional via heat} extends to $u\in H^s(\H^n)$ by a standard duality argument. 
	We also note that the fractional Laplacian on $\H^n$ can be characterized via a Caffarelli--Silvestre type extension; see \cite{BGS_15} for details.
	
	\subsection{Fractional Sobolev spaces}
	
	We use the notion of fractional Sobolev spaces introduced in \cite[Section 2]{BGS_15} and \cite[Section 3]{Tataru_strichartz}. 
	For \(a\in \R\), we define the \(L^2\)-based fractional Sobolev space on \(\H^n\) by
	\begin{equation}\label{eq:def of fractional Sobolev space}
		H^a(\H^n)
		:= \big\{ f\in \mathcal{D}'(\H^n):\ (\Id-\Delta_{\H^n})^{a/2}f \in L^2(\H^n)\big\},
	\end{equation}
	equipped with the norm
	\begin{equation}\label{eq:graph norm of H^a}
		\|f\|_{H^a(\H^n)}
		:= \|(\Id-\Delta_{\H^n})^{a/2}f\|_{L^2(\H^n)}.
	\end{equation}
	By the spectral calculus for the self-adjoint operator \(-\Delta_{\H^n}\), this is a Hilbert space. 
	Moreover, when \(a\ge 0\), one has the equivalent characterization
	\[
	H^a(\H^n)
	= \big\{ f\in L^2(\H^n):\ (-\Delta_{\H^n})^{a/2}f\in L^2(\H^n)\big\},
	\]
	and the norm \(\|f\|_{H^a(\H^n)}\) is equivalent to
	\begin{equation}\label{eq: graph norm of H^a}
		\|f\|_{L^2(\H^n)}+\|(-\Delta_{\H^n})^{a/2}f\|_{L^2(\H^n)}.
	\end{equation}
	
	Let \(\mathcal{O}\subset \H^n\) be a nonempty open set. 
	We denote by \(C_c^\infty(\mathcal{O})\) the space of smooth functions compactly supported in \(\mathcal{O}\). 
	For \(a\in \R\), we define
	\begin{align*}
		H^a(\mathcal{O}) &:= \{u|_{\mathcal{O}}:\ u\in H^a(\H^n)\},\\
		\wt H^a(\mathcal{O}) &:= \overline{C_c^\infty(\mathcal{O})}^{\,H^a(\H^n)}.
	\end{align*}
	The space \(H^a(\mathcal{O})\) is equipped with the quotient norm
	\[
	\|u\|_{H^a(\mathcal{O})}
	:= \inf\bigl\{ \|w\|_{H^a(\H^n)}:\ w\in H^a(\H^n)\ \text{and }\ w|_{\mathcal{O}}=u\bigr\}.
	\]
	Moreover, \(\wt H^a(\mathcal{O})\) is naturally identified with the subspace of \(H^a(\H^n)\) consisting of distributions supported in \(\overline{\mathcal{O}}\).
	
	In particular, when \(s\in(0,1)\), we define \(H^{-s}(\mathcal{O})\) as the dual space of \(\wt H^s(\mathcal{O})\), and \(\wt H^{-s}(\mathcal{O})\) as the dual space of \(H^s(\mathcal{O})\). That is,
	\[
	(\wt H^s(\mathcal{O}))^\ast = H^{-s}(\mathcal{O}),
	\qquad
	(H^s(\mathcal{O}))^\ast = \wt H^{-s}(\mathcal{O}),
	\]
	with respect to the duality pairing extending the \(L^2(\mathcal{O})\)-inner product.
	
	Finally, for any measurable subset \(D\subset \H^n\), we write
	\[
	(f,g)_{L^2(D)}:=\int_D f g \, dV.
	\]
	
	\subsection{The well-posedness}
	
	In this subsection, we study the well-posedness of the exterior value problem \eqref{eq: main}. 
	We begin with a mapping property of fractional powers.
	
	\begin{lemma}
		For $a\ge 0$, the operator $(-\Delta_{\H^n})^a$ extends to a bounded linear map
		\begin{equation}\label{eq: bounded linear map}
			(-\Delta_{\H^n})^a : H^r(\H^n) \to H^{r-2a}(\H^n),
		\end{equation}
		for any $r\in \R$. Moreover, $H^{a}(\H^n) \subset H^{b}(\H^n)$ for any $a\ge b$.
	\end{lemma}
	
	\begin{proof}
		Let $A:= \mathrm{Id}-\Delta_{\H^n}$. Since $A$ is a polynomial in $-\Delta_{\H^n}$, the operators commute under functional calculus. 
		In what follows, let us set 
		$$
		\tau:=\lambda^2+\frac{(n-1)^2}{4}\ge 0.
		$$ 
		Using \eqref{eq: graph norm of H^a} and the Plancherel identity \eqref{eq: Plancherel formula}, we compute
		\begin{equation}
			\begin{split}
				\norm{(-\Delta_{\H^n})^{a} f}_{H^{r-2a}(\H^n)}^2
				&= \norm{A^{(r-2a)/2} (-\Delta_{\H^n})^{a} f}_{L^2(\H^n)}^2\\
				&=\int_{\R\times \mathbb{S}^{n-1}}
				\big| \mathcal{F}_{\mathcal{H}}(A^{(r-2a)/2} (-\Delta_{\H^n})^{a} f)(\lambda,\theta)\big|^2 \, \frac{d\theta\, d\lambda}{|c(\lambda)|^2} \\
				&= \int_{\R \times \mathbb{S}^{n-1}}
				(1+\tau)^{r-2a}\tau^{2a}\big| \mathcal{F}_{\mathcal{H}}(f)(\lambda,\theta) \big|^2 \, \frac{d\theta\, d\lambda}{|c(\lambda)|^2} \\
				&\le \int_{\R \times \mathbb{S}^{n-1}}
				(1+\tau)^{r}\big| \mathcal{F}_{\mathcal{H}}(f)(\lambda,\theta) \big|^2 \, \frac{d\theta\, d\lambda}{|c(\lambda)|^2},
			\end{split}
		\end{equation}
		where we used the elementary inequality
		$(1+\tau)^{r-2a}\tau^{2a}\le (1+\tau)^r$. 
		Hence,
		$\norm{(-\Delta_{\H^n})^a f}_{H^{r-2a}(\H^n)} \le \norm{A^{r/2} f}_{L^2(\H^n)}=\norm{f}_{H^r(\H^n)}$,
		which proves \eqref{eq: bounded linear map}.
		
		Let $a\ge b$ and let $f\in H^{a}(\H^n)$. 
		By Plancherel \eqref{eq: Plancherel formula} and the definition of the Sobolev norm via functional calculus,
		\begin{equation*}
			\|f\|_{H^{b}(\H^n)}^{2}
			\sim 
			\int_{\R\times \mathbb S^{n-1}} (1+\tau)^{b}\,
			\big|\mathcal F_{\mathcal H}(f)(\lambda,\theta)\big|^{2}\,
			\frac{d\theta\, d\lambda}{|c(\lambda)|^{2}} .
		\end{equation*}
		Since $a\ge b$ and $1+\tau\ge 1$, we have the pointwise inequality
		\[
		(1+\tau)^{b}\le (1+\tau)^{a}.
		\]
		Therefore,
		\begin{equation*}
			\|f\|_{H^{b}(\H^n)}^{2}
			\lesssim
			\int_{\R\times \mathbb S^{n-1}} (1+\tau)^{a}\,
			\big|\mathcal F_{\mathcal H}(f)(\lambda,\theta)\big|^{2}\,
			\frac{d\theta\, d\lambda}{|c(\lambda)|^{2}}
			\simeq
			\|f\|_{H^{a}(\H^n)}^{2}.
		\end{equation*}
		Hence $\|f\|_{H^{b}(\H^n)}\le C\|f\|_{H^{a}(\H^n)}$ for some constant $C>0$, and in particular
		$H^{a}(\H^n)\subset H^{b}(\H^n)$ continuously.
	\end{proof}
	
	Next, let $\Omega\subset \H^n$ be a bounded open set and consider
	\begin{equation}\label{eq: equation of well-posed}
		\begin{cases}
			\big( \mathcal{P}_{\H^n}+ q \big) u =F & \text{ in }\Omega,\\
			u=f &\text{ in }\Omega_e,
		\end{cases}
	\end{equation}
	where $\mathcal{P}_{\H^n}$ is the fractional polyharmonic operator given by \eqref{eq: fractional poly op.}.
	Define the bilinear form
	\begin{equation}\label{eq: bilinear form}
		\begin{split}
			B_q (u,w)
			&:= \sum_{k=1}^N b_k \big(  (-\Delta_{\H^n})^{s_k/2} u,\ (-\Delta_{\H^n})^{s_k/2} w \big)_{L^2(\H^n)}
			+ ( q u, w)_{L^2(\Omega)}.
		\end{split}
	\end{equation}
	
	\begin{lemma}[Well-posedness of the exterior Dirichlet problem]\label{lem: well-posedness}
		Let $\Omega \Subset \mathbb H^n$ be a bounded domain with Lipschitz boundary and $q \in L^\infty(\Omega)$. 
		Let $\mathcal{P}_{\H^n}$ be the operator given by \eqref{eq: fractional poly op.}. 
		Assume that \eqref{eq: eigenvalue condition} holds, i.e., the only solution $u\in \widetilde H^{s_N}(\Omega)$ of $(\mathcal{P}_{\mathbb H^n}+q)u=0$ in $\Omega$ is $u\equiv 0$. 
		Then for any $f\in H^{s_N}(\mathbb H^n)$ and $F\in (\widetilde H^{s_N}(\Omega))^*$, the problem
		\[
		\begin{cases}
			(\mathcal{P}_{\mathbb H^n}+q)u=F & \text{in }\Omega,\\
			u=f & \text{in }\Omega_e,
		\end{cases}
		\]
		admits a unique solution $u\in H^{s_N}(\mathbb H^n)$. Moreover,
		\begin{equation}\label{eq: solution_estimate_general}
			\|u\|_{H^{s_N}(\mathbb H^n)}
			\le C\big(\|f\|_{H^{s_N}(\mathbb H^n)} + \|F\|_{(\widetilde H^{s_N}(\Omega))^*}\big),
		\end{equation}
		for some constant $C>0$ independent of $u,f,$ and $F$.
	\end{lemma}
	
	\begin{proof}
		Write $u=v+f$ with $v\in \widetilde H^{s_N}(\Omega)$. 
		Then $v$ satisfies the variational problem
		\[
		B_q(v,w)=\widetilde F(w),
		\quad \text{for all } w\in \widetilde H^{s_N}(\Omega),
		\]
		where $\widetilde F(w):=F(w)-B_q(f,w)$. 
		The form $B_q$ is bounded on $\widetilde H^{s_N}(\Omega)$.
		
		Let $q_-(x)=\max\{-q(x),0\}$ and choose $\mu>\|q_-\|_{L^\infty(\Omega)}$. 
		Then for $v\in \widetilde H^{s_N}(\Omega)$,
		\[
		B_q(v,v)+\mu\|v\|_{L^2(\Omega)}^2
		\ge
		b_N\|(-\Delta_{\mathbb H^n})^{s_N/2}v\|_{L^2(\mathbb H^n)}^2
		+(\mu-\|q_-\|_{L^\infty(\Omega)})\|v\|_{L^2(\Omega)}^2,
		\]
		hence
		\[
		B_q(v,v)+\mu\|v\|_{L^2(\Omega)}^2
		\ge c \|v\|_{\widetilde H^{s_N}(\Omega)}^2
		\]
		for some $c>0$. 
		By the Lax--Milgram theorem, there exists a bounded operator
		\[
		G_\mu : (\widetilde H^{s_N}(\Omega))^\ast \to \widetilde H^{s_N}(\Omega)
		\]
		such that
		\[
		B_q(v,w)+\mu(v,w)_{L^2(\Omega)}=\ell(w)
		\quad \text{for all }w\in \widetilde H^{s_N}(\Omega).
		\]
		
		Since $\Omega\subset \H^n$ is bounded and $H^n$ has bounded geometry, the geometry is locally comparable to Euclidean space. 
		In particular, $\widetilde H^{s_N}(\Omega)\hookrightarrow L^2(\Omega)$ is compact by the classical Rellich--Kondrachov theorem.
		Let $\iota:\widetilde H^{s_N}(\Omega)\hookrightarrow L^2(\Omega)$ be the embedding and set
		\[
		K_\mu:=\iota \circ G_\mu \circ \iota^\ast 
		: L^2(\Omega)\to L^2(\Omega).
		\]
		Since $\iota : \widetilde H^{s_N}(\Omega) \hookrightarrow L^2(\Omega)$ is compact, the operator $K_\mu$ is compact on $L^2(\Omega)$.
		
		The original problem is equivalent to the Fredholm equation of index zero
		\[
		(\mathrm{Id}+\mu K_\mu)v
		=
		G_\mu \widetilde F
		\quad \text{in } L^2(\Omega).
		\]
		By the assumption that $0$ is not a Dirichlet eigenvalue of $\mathcal{P}_{\H^n}+q$, the homogeneous equation
		\[
		(\mathcal{P}_{\mathbb H^n}+q)v=0,
		\quad v\in \widetilde H^{s_N}(\Omega),
		\]
		admits only the trivial solution, hence the Fredholm operator has a trivial kernel and is invertible. 
		Therefore, $v$ (and thus $u=v+f$) exists uniquely, and the estimate \eqref{eq: solution_estimate_general} follows from the preceding bounds.
	\end{proof}
	
	\noindent
	Note that Assumption~\ref{exponent condition} is not needed to show the well-posedness result in Lemma~\ref{lem: well-posedness}.
	
	\subsection{The DN map}\label{subsec:DNmap}
	
	Let us define the DN map associated with the exterior value problem \eqref{eq: main}, where $q\in L^\infty(\Omega)$ satisfies \eqref{eq: eigenvalue condition}. 
	We introduce the abstract trace space
	\[
	X := H^{s_N}(\mathbb H^n)\big/ \widetilde H^{s_N}(\Omega),
	\]
	where $\widetilde H^{s_N}(\Omega)$ denotes the closure of $C_c^\infty(\Omega)$ in $H^{s_N}(\mathbb H^n)$.
	We write $[f]\in X$ for the equivalence class of $f\in H^{s_N}(\mathbb H^n)$.
	
	\begin{lemma}[DN map on $\mathbb H^n$]\label{lem:DNmap_hyp}
		Let $\Omega\subset \mathbb H^n$ be a bounded domain, let $0<s_1<\cdots<s_N$, and let
		$q\in L^\infty(\Omega)$ satisfy \eqref{eq: eigenvalue condition}.
		Then there exists a bounded linear map
		\[
		\Lambda_q : X \to X^*,
		\]
		defined by
		\begin{equation}\label{eq:DN_def_hyp}
			\langle \Lambda_q [f], [g]\rangle := B_q(u_f, g),
			\qquad f,g\in H^{s_N}(\mathbb H^n),
		\end{equation}
		where $u_f\in H^{s_N}(\mathbb H^n)$ is the unique solution of \eqref{eq: main} such that
		$u_f-f\in \widetilde H^{s_N}(\Omega)$.
		Moreover, $\Lambda_q$ is symmetric in the sense that
		\begin{equation}\label{eq:DN_sym_hyp}
			\langle \Lambda_q [f], [g]\rangle=\langle \Lambda_q [g], [f]\rangle,
			\qquad f,g\in H^{s_N}(\mathbb H^n).
		\end{equation}
	\end{lemma}
	
	\begin{proof}
		\textit{Step 1: Well-definedness with respect to the quotient representatives.}
		Let \(f,g\in H^{s_N}(\mathbb H^n)\). We first show that the quantity
		\[
		B_q(u_f,g)
		\]
		depends only on the equivalence classes \([f],[g]\in X\).
		
		To see that \(u_f\) depends only on \([f]\), let \(\varphi\in \widetilde H^{s_N}(\Omega)\) and set
		\[
		f':=f+\varphi.
		\]
		Since \(u_f-f\in \widetilde H^{s_N}(\Omega)\), we have
		\[
		u_f-f'=(u_f-f)-\varphi \in \widetilde H^{s_N}(\Omega).
		\]
		Hence \(u_f\) is also the weak solution corresponding to the exterior datum \(f'\). By the uniqueness of weak solutions, we obtain
		\[
		u_{f'}=u_f.
		\]
		Therefore, \(u_f\) depends only on the class \([f]\).
		
		Next, let \(\psi\in \widetilde H^{s_N}(\Omega)\) and set
		\[
		g':=g+\psi.
		\]
		Since \(u_f\) solves
		\[
		(\mathcal P_{\mathbb H^n}+q)u_f=0 \quad \text{in }\Omega
		\]
		in the weak sense, we have
		\[
		B_q(u_f,\eta)=0 \qquad \text{for all }\eta\in \widetilde H^{s_N}(\Omega).
		\]
		In particular,
		\[
		B_q(u_f,\psi)=0.
		\]
		By bilinearity, it follows that
		\[
		B_q(u_f,g')=B_q(u_f,g+\psi)=B_q(u_f,g)+B_q(u_f,\psi)=B_q(u_f,g).
		\]
		Thus \(B_q(u_f,g)\) depends only on \([g]\).
		
		Combining the two observations above, we conclude that \eqref{eq:DN_def_hyp} depends only on \([f]\) and \([g]\). Hence \(\Lambda_q\) is well defined.
		
		\smallskip
		\textit{Step 2: Boundedness.}
		By Cauchy--Schwarz and the definition of $B_q$,
		\[
		|B_q(u_f,g)|
		\le
		\sum_{k=1}^N b_k
		\|(-\Delta_{\mathbb H^n})^{s_k/2}u_f\|_{L^2(\mathbb H^n)}
		\|(-\Delta_{\mathbb H^n})^{s_k/2}g\|_{L^2(\mathbb H^n)}
		+\|q\|_{L^\infty(\Omega)}\|u_f\|_{L^2(\Omega)}\|g\|_{L^2(\Omega)}.
		\]
		Since $s_k\le s_N$, we have continuous embeddings
		$\|(-\Delta_{\mathbb H^n})^{s_k/2}h\|_{L^2}\le C \|h\|_{H^{s_N}(\mathbb H^n)}$ and
		$\|h\|_{L^2(\Omega)}\le C\|h\|_{H^{s_N}(\mathbb H^n)}$.
		Therefore
		\[
		|B_q(u_f,g)|
		\le C \|u_f\|_{H^{s_N}(\mathbb H^n)}\|g\|_{H^{s_N}(\mathbb H^n)}.
		\]
		By \eqref{eq: solution_estimate_general} with $F=0$,
		$\|u_f\|_{H^{s_N}(\mathbb H^n)}\le C\|f\|_{H^{s_N}(\mathbb H^n)}$.
		Thus
		\[
		|\langle \Lambda_q[f],[g]\rangle|
		\le C \|f\|_{H^{s_N}(\mathbb H^n)}\|g\|_{H^{s_N}(\mathbb H^n)}.
		\]
		Taking the infimum over representatives yields
		\[
		|\langle \Lambda_q[f],[g]\rangle|
		\le C \|[f]\|_{X}\,\|[g]\|_{X},
		\]
		so $\Lambda_q:X\to X^*$ is bounded.
		
		\smallskip
		\textit{Step 3: Symmetry.}
		Let $u_f,u_g$ be the corresponding solutions.
		Using that $g-u_g\in \widetilde H^{s_N}(\Omega)$ and $f-u_f\in \widetilde H^{s_N}(\Omega)$, and that
		$B_q(u_f,\eta)=B_q(u_g,\eta)=0$ for all $\eta\in \widetilde H^{s_N}(\Omega)$, we obtain
		\[
		B_q(u_f,g)=B_q(u_f,u_g)=B_q(u_g,u_f)=B_q(u_g,f),
		\]
		which proves \eqref{eq:DN_sym_hyp}.
	\end{proof}
	
	\begin{remark}
		If $\partial\Omega$ is Lipschitz, one may identify the quotient space
		$X=H^{s_N}(\mathbb H^n)/\widetilde H^{s_N}(\Omega)$ with a Sobolev space on the exterior domain
		$\Omega_e$ (with identifications depending on the range of $s_N$). In the present work, the abstract quotient definition is sufficient. For simplicity, we write $f$ in place of $[f]$ in the remainder of the paper.
	\end{remark}
	
	\begin{proposition}[Characterization of the DN map]\label{prop:DN_characterization}
		Let $f\in H^{s_N}(\mathbb H^n)$ and let $u_f\in H^{s_N}(\mathbb H^n)$ be the unique solution of
		\[
		\begin{cases}
			(\mathcal P_{\mathbb H^n}  + q)  u_f = 0 & \text{ in }\Omega,\\
			u_f = f & \text{ in }\Omega_e.
		\end{cases}
		\]
		Then the DN map satisfies
		\[
		\Lambda_q f
		=
		\big(\mathcal P_{\mathbb H^n} u_f\big)\big|_{\Omega_e}
		\]
		in the sense that for all $g\in H^{s_N}(\mathbb H^n)$ with $\supp(g)\subset \Omega_e$,
		\begin{equation}\label{eq:DN_characterization}
			\langle \Lambda_q f, g \rangle
			=
			\int_{\Omega_e}
			(\mathcal P_{\mathbb H^n} u_f)\, g \, dV.
		\end{equation}
	\end{proposition}
	
	\begin{proof}
		Let $g\in H^{s_N}(\mathbb H^n)$ with $\supp(g)\subset\Omega_e$.
		By definition,
		\[
		\langle \Lambda_q f, g \rangle = B_q(u_f,g).
		\]
		Since $g$ vanishes in $\Omega$, the potential term disappears and
		\[
		B_q(u_f,g)
		=
		\sum_{k=1}^N b_k
		\big(
		(-\Delta_{\mathbb H^n})^{s_k/2}u_f,
		(-\Delta_{\mathbb H^n})^{s_k/2}g
		\big)_{L^2(\mathbb H^n)}.
		\]
		Using the self-adjointness of $(-\Delta_{\H^n})^{s_k/2}$ on $L^2(\H^n)$ (equivalently, duality),
		\[
		\big(
		(-\Delta_{\mathbb H^n})^{s_k/2}u_f,
		(-\Delta_{\mathbb H^n})^{s_k/2}g
		\big)_{L^2(\H^n)}
		=
		\big\langle
		(-\Delta_{\mathbb H^n})^{s_k}u_f,\ g
		\big\rangle_{H^{-s_k}(\H^n)\times H^{s_k}(\H^n)}.
		\]
		Summing over $k$ yields
		\[
		B_q(u_f,g)=\int_{\H^n}(\mathcal P_{\mathbb H^n}u_f)\, g\, dV.
		\]
		Since $g$ is supported in $\Omega_e$, this reduces to \eqref{eq:DN_characterization}.
	\end{proof}
	
	\begin{lemma}[Integral identity]\label{lem:integral_identity_hyp}
		Let $\Omega \Subset \mathbb H^n$ be a bounded domain and let $q_1,q_2\in L^\infty(\Omega)$ satisfy \eqref{eq: eigenvalue condition}.
		Let $\Lambda_{q_1}$ and $\Lambda_{q_2}$ be the corresponding DN maps. Then for any $f,g\in H^{s_N}(\mathbb H^n)$,
		\begin{equation}\label{eq: integral_identity_hyp}
			\langle (\Lambda_{q_1}-\Lambda_{q_2}) f, g\rangle
			=
			\int_{\Omega} (q_1-q_2)u_1 u_2 \, dV,
		\end{equation}
		where $u_j\in H^{s_N}(\mathbb H^n)$ solves
		\begin{equation}\label{eq: main j=1,2}
			\begin{cases}
				\big( \mathcal{P}_{\H^n}+q_j \big) u_j=0 &\text{ in }\Omega , \\
				u_1=f \quad  \text{and}\quad u_2=g &\text{ in }\Omega_e,
			\end{cases}
		\end{equation}
		for $j=1,2$.
	\end{lemma}
	
	\begin{proof}
		By definition of the DN map,
		\[
		\langle \Lambda_{q_1} f, g\rangle
		=
		B_{q_1}(u_1,g),
		\qquad
		\langle \Lambda_{q_2} f, g\rangle
		=
		B_{q_2}(u_2,g).
		\]
		Since $u_2-g\in \widetilde H^{s_N}(\Omega)$ and $u_1$ solves $(\mathcal{P}_{\H^n}+q_1)u_1=0$ in $\Omega$, we have
		\[
		B_{q_1}(u_1,u_2-g)=0,
		\]
		and therefore $B_{q_1}(u_1,g)=B_{q_1}(u_1,u_2)$. Similarly, $B_{q_2}(u_2,g)=B_{q_2}(u_2,u_1)$. 
		Using symmetry of $B_q$, we obtain
		\[
		\langle (\Lambda_{q_1}-\Lambda_{q_2}) f, g\rangle
		=
		B_{q_1}(u_1,u_2)-B_{q_2}(u_2,u_1).
		\]
		Expanding $B_{q_j}$, the fractional Laplacian terms cancel, and we arrive at
		\[
		B_{q_1}(u_1,u_2)-B_{q_2}(u_2,u_1)
		=
		\int_\Omega (q_1-q_2) u_1 u_2 \, dV,
		\]
		which proves \eqref{eq: integral_identity_hyp}.
	\end{proof}

	\section{Entanglement principle}\label{sec:entangle}
	
	In this section, we prove Theorem~\ref{thm: ent}. By reducing integer parts of the exponents, it suffices to establish the result in the case $\{s_k\}_{k=1}^N\subset (0,1)$. Let $e_0=(1,0,\dots,0)\in \mathbb H^n$. We begin with the following smooth, rapidly decaying version.
	
	\begin{lemma}[Entanglement principle for smooth functions]\label{lem: ent smooth}
		Let \(O \subset \H^n\) be a nonempty open set with \(n \ge 2\). Let \(N\in \N\) and
		\[
		0<\alpha_1<\alpha_2<\cdots<\alpha_N<1.
		\]
		For \(\gamma>\frac{n-1}{2}\), define
		\[
		\rho_\gamma(x):=e^{-\gamma\sqrt{1+d_{\H^n}(e_0,x)^2}}, \qquad x\in \H^n.
		\]
		Assume that \(\{v_k\}_{k=1}^N\subset C^\infty(\H^n)\) satisfies the exponential decay condition: for every multi-index \(\beta\) there exists \(C_\beta>0\) such that
		\begin{equation}\label{eq: exponential decay condition at infinity}
			|\nabla^\beta v_k(x)| \le C_\beta \rho_\gamma(x)
			\qquad \text{for all }x\in \H^n,\ \ k=1,\dots,N.
		\end{equation}
		If
		\begin{equation}\label{condition_entanglement_w_reduced}
			\begin{split}
				v_1|_O=\cdots=v_N|_O=0
				\qquad\text{and}\qquad
				\left.\left(\sum_{k=1}^N (-\Delta_{\H^n})^{\alpha_k}v_k\right)\right|_O=0,
			\end{split}
		\end{equation}
		then
		\[
		v_k\equiv 0 \qquad \text{in }\H^n,\quad k=1,\dots,N.
		\]
	\end{lemma}
	
	\begin{remark}
		Here $\nabla^\beta$ denotes iterated covariant derivatives with respect to the Levi--Civita connection of $\H^n$.
		In local coordinates, one can write
		\[
		\nabla^\beta
		=
		\sum_{|\alpha|\le |\beta|}
		a_{\alpha}(x)\,\partial^\alpha,
		\]
		where the coefficients $a_{\alpha}\in C^\infty(\H^n)$ depend on the metric and its Christoffel symbols.
		Since $\H^n$ has bounded geometry, these coefficients are uniformly controlled; in particular, pointwise decay for coordinate derivatives implies the same decay for covariant derivatives.
		Moreover, since $\gamma > \frac{n-1}{2}$, the decay \eqref{eq: exponential decay condition at infinity} implies $\nabla^\beta v_k\in L^2(\H^n)$ for all $\beta$.
	\end{remark}
	
	Using Lemma~\ref{lem: ent smooth}, we now prove Theorem~\ref{thm: ent}.
	
	\begin{proof}[Proof of Theorem~\ref{thm: ent}]
		For each $k=1,\dots,N$, write
		\[
		s_k=m_k+\alpha_k,
		\qquad
		m_k:=\lfloor s_k\rfloor\in \mathbb N\cup\{0\},
		\quad
		\alpha_k\in(0,1).
		\]
		By Assumption~\ref{exponent condition}, $\alpha_k\neq \alpha_j$ for $j\neq k$. Reorder indices if necessary so that
		\[
		0<\alpha_1<\alpha_2<\cdots<\alpha_N<1.
		\]
		
		\medskip
		
		{\it Step 1: mollification and pointwise exponential decay.}
		Since \(u_k \in H^{-r}(\H^n)\) need not be smooth,
		we regularize by radial convolution. Choose a nonempty open set \(O_0 \Subset O\) and \(\delta > 0\) such that
		\[
		d_{\H^n}(O_0,\H^n\setminus O)\ge \delta.
		\]
		Pick \(\eta\in C_c^\infty([0,\infty))\) with \(\eta\ge 0\), \(\supp \eta \subset [0,1)\), and \(\eta\not\equiv 0\). For \(\varepsilon\in (0,\delta)\) define the radial mollifier
		by
		\[
		\varphi_\varepsilon(d_{\H^n}(x,y))
		:=
		\frac{1}{A_\varepsilon}\eta\big(\varepsilon^{-1}d_{\H^n}(x,y)\big),
		\qquad
		A_\varepsilon:=\int_{\H^n}\eta\big(\varepsilon^{-1}d_{\H^n}(x,y)\big)\,dV(y).
		\]
		By homogeneity of \(\H^n\), the normalization constant \(A_\varepsilon\) is independent of the center point \(x\) and
		depends only on \(\varepsilon\). In particular,
		\[
		\varphi_\varepsilon(d_{\H^n}(x,\cdot)) \in C_c^\infty(\H^n), \qquad
		\supp(\varphi_\varepsilon(d_{\H^n}(x,\cdot))) \subset B_\varepsilon(x), \qquad
		\|\varphi_\varepsilon(d_{\H^n}(x,\cdot))\|_{L^1(\H^n)}=1.
		\]
		Moreover, for each multi-index \(\beta\), the quantity
		\begin{equation}\label{eq:mollifier_uniform_bound}
			\big\|\nabla_x^\beta(\varphi_\varepsilon(d_{\H^n}(x,\cdot)))\big\|_{H^r(\H^n)}
		\end{equation}
		is bounded uniformly in \(x\) (by invariance under isometries).
		
		Define, for each \(k=1,\dots,N\),
		\[
		u_{k,\varepsilon}(x):=(u_k\ast \varphi_\varepsilon)(x):=
		\big\langle u_k,\varphi_\varepsilon(d_{\H^n}(x,\cdot))\big\rangle,
		\]
		then \(u_{k,\varepsilon}\in C^\infty(\H^n)\). For any multi-index \(\beta\), differentiation under the duality pairing gives
		\[
		\nabla^\beta u_{k,\varepsilon}(x)
		=
		\big\langle u_k,\nabla_x^\beta(\varphi_\varepsilon(d_{\H^n}(x,\cdot)))\big\rangle.
		\]
		Using \eqref{eq: exponential decay condition weakly} with
		\(\phi=\nabla_x^\beta(\varphi_\varepsilon(d_{\H^n}(x,\cdot)))\), we obtain
		\[
		|\nabla^\beta u_{k,\varepsilon}(x)|
		\le C\,
		\big\|\rho_\gamma \nabla_x^\beta(\varphi_\varepsilon(d_{\H^n}(x,\cdot)))\big\|_{H^r(\H^n)}.
		\]
		Since \(\supp(\varphi_\varepsilon(d_{\H^n}(x,\cdot)))\subset B_\varepsilon(x)\), we have
		\[
		\big|\sqrt{1+d_{\H^n}(e_0,y)^2}-\sqrt{1+d_{\H^n}(e_0,x)^2}\big|
		\le d_{\H^n}(x,y)\le \varepsilon
		\]
		for \(y\in \supp(\varphi_\varepsilon(d_{\H^n}(x,\cdot)))\), because the map
		\[
		z\mapsto \sqrt{1+d_{\H^n}(e_0,z)^2}
		\]
		is \(1\)-Lipschitz on \(\H^n\). Hence
		\[
		\rho_\gamma(y)\le e^{\gamma\varepsilon}\rho_\gamma(x)
		\qquad \text{for all }y\in \supp(\varphi_\varepsilon(d_{\H^n}(x,\cdot))).
		\]
		Therefore, using \eqref{eq:mollifier_uniform_bound}, there exists \(C_{\beta,\varepsilon}>0\) such that
		\begin{equation}\label{eq:ukeps_decay}
			|\nabla^\beta u_{k,\varepsilon}(x)|
			\le C_{\beta,\varepsilon}\rho_\gamma(x),
			\qquad x\in \H^n,\ \ k=1,\dots,N.
		\end{equation}
		In particular, since \(\gamma>\frac{n-1}{2}\) and the hyperbolic volume growth is of order \(e^{(n-1)r}\),
		we have \(\nabla^\beta u_{k,\varepsilon}\in L^2(\H^n)\) for each fixed \(\beta\).

		\medskip
		
		{\it Step 2: the equation is preserved on $\mathcal{O}'$.}
		Since $\varphi_\varepsilon$ is radial, convolution with $\varphi_\varepsilon$ commutes with functional calculus of $-\Delta_{\H^n}$; in particular,
		\[
		(-\Delta_{\H^n})^{s}(u_k*\varphi_\varepsilon)
		=\big((-\Delta_{\H^n})^{s}u_k\big)*\varphi_\varepsilon
		\quad\text{in }\mathcal D'(\H^n).
		\]
		Since $u_k|_{\mathcal O}=0$ in the distributional sense and $0<\varepsilon<\delta$, we have $u_{k,\varepsilon}|_{\mathcal{O}'}=0$ (because $B_\varepsilon(x)\subset\mathcal O$ for $x\in\mathcal O'$).
		Convolving $\sum_{k=1}^N b_k(-\Delta_{\H^n})^{s_k}u_k=0$ in $\mathcal{O}$ with $\varphi_\varepsilon$ yields
		\[
		\sum_{k=1}^N b_k(-\Delta_{\H^n})^{s_k}u_{k,\varepsilon}=0
		\qquad\text{in }\mathcal{O}'.
		\]
		Hence,
		\[
		u_{1,\varepsilon}|_{\mathcal{O}'}=\cdots=u_{N,\varepsilon}|_{\mathcal{O}'}=0,
		\qquad
		\bigg(\sum_{k=1}^N b_k(-\Delta_{\H^n})^{s_k}u_{k,\varepsilon}\bigg)\Big|_{\mathcal{O}'}=0.
		\]
		
		\medskip
		
		{\it Step 3: reduction to exponents in $(0,1)$ and application of Lemma~\ref{lem: ent smooth}.}
		Define
		\[
		v_{k,\varepsilon}:=b_k(-\Delta_{\H^n})^{m_k}u_{k,\varepsilon}\in C^\infty(\H^n),
		\qquad k=1,\dots,N.
		\]
		Then on $\mathcal O'$,
		\[
		v_{1,\varepsilon}|_{\mathcal O'}=\cdots=v_{N,\varepsilon}|_{\mathcal O'}=0,
		\qquad
		\sum_{k=1}^N (-\Delta_{\H^n})^{\alpha_k}v_{k,\varepsilon}\Big|_{\mathcal O'}=0,
		\]
		since $(-\Delta_{\H^n})^{s_k}=(-\Delta_{\H^n})^{\alpha_k}(-\Delta_{\H^n})^{m_k}$.
		Moreover, $(-\Delta_{\H^n})^{m_k}$ is a local differential operator of order $2m_k$, so $\nabla^\beta v_{k,\varepsilon}$ is a finite linear combination of covariant derivatives of $u_{k,\varepsilon}$.
		Thus, \eqref{eq:ukeps_decay} implies that for every $\beta$ there exists $C'_{\beta,\varepsilon}$ such that
		\[
		|\nabla^\beta v_{k,\varepsilon}(x)|
		\le C'_{\beta,\varepsilon}\,\rho_\gamma(x),
		\qquad x\in\H^n.
		\]
		Therefore Lemma~\ref{lem: ent smooth} applies (with $\mathcal O'$ in place of $\mathcal O$), and we conclude
		\[
		v_{k,\varepsilon}\equiv 0 \quad\text{in }\H^n \quad\text{for each }k=1,\dots,N.
		\]
		Equivalently,
		\[
		(-\Delta_{\H^n})^{m_k}u_{k,\varepsilon}\equiv 0
		\quad\text{in }\H^n,\quad k=1,\dots,N.
		\]
		
		\medskip
		
		{\it Step 4: elliptic unique continuation.}
		Since $(-\Delta_{\H^n})^{m_k} u_{k,\varepsilon}=0$ in $\H^n$ and $u_{k,\varepsilon}$ vanishes on a nonempty open set $\mathcal{O}'$, standard elliptic unique continuation for $\Delta_{\H^n}$ implies $u_{k,\varepsilon}\equiv 0$ in $\H^n$.
		
		\medskip
		
		{\it Step 5: letting $\varepsilon\to 0$.}
		Since $\{\varphi_\varepsilon\}_{\varepsilon>0}$ is an approximate identity on $\H^n$ and $u_k\in H^{-r}(\H^n)$, we have
		$u_{k,\varepsilon}\to u_k$ in $H^{-r}(\H^n)$ as $\varepsilon\to 0$.
		Because $u_{k,\varepsilon}\equiv 0$ for all sufficiently small $\varepsilon$, it follows that $u_k\equiv 0$ in $H^{-r}(\H^n)$, hence $u_k\equiv 0$ in $\H^n$ for each $k=1,\dots,N$.
	\end{proof}
	
	It remains to prove Lemma~\ref{lem: ent smooth}. We first recall the following decoupling criterion investigated in \cite{FKU24}.
	
	\begin{proposition}[\text{\cite[Proposition 3.1]{FKU24}}]\label{prop: key of entangle}
		Let $N\in \N$ and $0<\alpha_1 <\ldots < \alpha_N<1$ satisfy Assumption~\ref{exponent condition}.
		Suppose $\{f_k\}_{k=1}^N\subset C^\infty((0,\infty))$ and there exist $c,\delta>0$ such that for each $k$
		\begin{equation}\label{key exponential decay condition}
			|f_k(t)|\le ce^{-\delta t}\quad (t\in (1,\infty)),
			\qquad
			|f_k(t)|\le ce^{-\delta/t}\quad (t\in (0,1]).
		\end{equation}
		If there exists $\ell\in \N\cup\{0\}$ such that
		\begin{equation}\label{FKU key id}
			\sum_{k=1}^N \Gamma(m+1+\alpha_k)\int_0^\infty f_k(t)\, t^{-m}\, dt =0
			\quad\text{for all }m=\ell,\ell+1,\ell+2,\dots,
		\end{equation}
		then $f_k(t)\equiv 0$ for $t\in (0,\infty)$ for all $k=1,\ldots,N$.
	\end{proposition}
	
	With Proposition \ref{prop: key of entangle} at hand, we can prove Lemma \ref{lem: ent smooth}.
	
	\begin{proof}[Proof of Lemma~\ref{lem: ent smooth}]
		We first note that \eqref{condition_entanglement_w_reduced} implies, for every $m\in\N$,
		\begin{equation}\label{eq: initial identity}
			\begin{split}
				\Delta_{\H^n}^{m}v_1=\cdots=\Delta_{\H^n}^{m}v_N=0
				\quad\text{and}\quad
				\sum_{k=1}^N (-\Delta_{\H^n})^{\alpha_k}(-\Delta_{\H^n})^{m} v_k=0
				\quad\text{in }\mathcal O.
			\end{split}
		\end{equation}
		Fix a nonempty bounded open set $\omega\Subset \mathcal O$ and choose $\kappa\in(0,1)$ such that
		\begin{equation}\label{eq: distance kappa}
			d_{\H^n}(\overline{\omega},\H^n\setminus \mathcal O)\ge \kappa.
		\end{equation}
		
		Using the heat semigroup representation \eqref{eq: fractional via heat}, \eqref{eq: initial identity} yields, for \(x\in\omega\) and every \(m\in\N\),
		\begin{equation}\label{pf of UCP 1}
			\sum_{k=1}^N \frac{1}{\Gamma(-\alpha_k)} \int_0^\infty
			\big(e^{t\Delta_{\H^n}}\Delta_{\H^n}^{m} v_k\big)(x)\,
			\frac{dt}{t^{1+\alpha_k}}=0.
		\end{equation}
		As in \cite[Lemma 3.7]{FKU24} and \cite[Proposition 3.6]{FL24}, this identity can be rewritten as
		\begin{equation}\label{eq_important_m}
			\sum_{k=1}^N \frac{\Gamma(1+m+\alpha_k)}{\Gamma(-\alpha_k)\Gamma(1+\alpha_k)}
			\int_0^\infty (e^{t\Delta_{\H^n}}v_k)(x)\, t^{-(1+m+\alpha_k)}\,dt =0
			\quad \text{for all }x\in\omega,\ \ m\in \N.
		\end{equation}
		For completeness, we sketch the integration-by-parts argument.
		Since \(e^{t\Delta_{\H^n}}\Delta_{\H^n}^{m}=\Delta_{\H^n}^{m}e^{t\Delta_{\H^n}}\) on \(H^{2m}(\H^n)\), we have
		\begin{equation}\label{eq_100_4}
			\big(e^{t\Delta_{\H^n}}\Delta_{\H^n}^{m} v_k\big)(x)=\partial_t^{m}\big(e^{t\Delta_{\H^n}}v_k\big)(x),
			\qquad x\in\omega,\ \ t>0.
		\end{equation}
		Thus \eqref{pf of UCP 1} becomes
		\begin{equation}\label{eq_100_5}
			\sum_{k=1}^N \frac{1}{\Gamma(-\alpha_k)}\int_0^\infty
			\partial_t^{m}(e^{t\Delta_{\H^n}}v_k)(x)\,
			\frac{dt}{t^{1+\alpha_k}}=0,
			\qquad x\in\omega,\ \ m\in\N.
		\end{equation}
		
		We next justify integration by parts \(m\) times. For \(0\le \ell\le m-1\), we have
		\begin{equation}\label{eq_100_6}
			\begin{split}
				\big|\partial_t^\ell (e^{t\Delta_{\H^n}}v_k)(x)\big|
				&=\big|(e^{t\Delta_{\H^n}}\Delta_{\H^n}^{\ell}v_k)(x)\big| \\
				&= \bigg|\int_{\H^n\setminus \overline{\mathcal{O}}} p_t(x,y)\, \Delta_{\H^n}^{\ell} v_k(y)\, dV(y)\bigg| \\
				& \le \bigg( \int_{\H^n\setminus \overline{\mathcal{O}}} p_t(x,y)^2 \, dV(y) \bigg)^{1/2}
				\big\| \Delta_{\H^n}^{\ell} v_k \big\|_{L^2(\H^n\setminus \overline{\mathcal{O}})} .
			\end{split}
		\end{equation}
		
		Using the global heat kernel bounds \eqref{eq: global bounds heat kernel} and the separation condition \eqref{eq: distance kappa}, one obtains
		\begin{equation}\label{eq_100_4_4}
			\int_{\{y:\ d_{\H^n}(x,y)\ge \kappa\}} p_t(x,y)^2\, dV(y)
			\le C\, t^{-n+\frac12}(1+t)^{\frac{3n-7}{2}}
			\exp\!\Big(-\frac{(n-1)^2}{2}t\Big)\exp\!\Big(-\frac{\kappa^2}{2t}\Big),
		\end{equation}
		for \(x\in\omega\) and \(t>0\).
		Moreover, by \eqref{eq: exponential decay condition at infinity} (and \(\gamma>(n-1)/2\)) we have
		\begin{equation}\label{eq_100_4_5}
			\big\|\Delta_{\H^n}^{\ell} v_k\big\|_{L^2(\H^n\setminus \overline{\mathcal{O}})}<\infty.
		\end{equation}
		Combining \eqref{eq_100_6}--\eqref{eq_100_4_5}, we obtain
		\begin{equation}\label{eq_100_7}
			\big|\partial_t^\ell (e^{t\Delta_{\H^n}}v_k)(x)\big|
			\le C\, t^{-\frac{n}{2}+\frac14}(1+t)^{\frac{3n-7}{4}}
			\exp\!\Big(-\frac{(n-1)^2}{4}t\Big)\exp\!\Big(-\frac{\kappa^2}{4t}\Big)
			\big\| \Delta_{\H^n}^{\ell} v_k \big\|_{L^2(\H^n\setminus \overline{\mathcal{O}})},
		\end{equation}
		for \(x\in\omega\), \(t>0\), and \(0\le \ell\le m-1\), where the constant \(C\) is independent of \(t\).
		
		In particular, for every \(0\le \ell\le m-1\) and every \(\beta\ge 0\),
		\[
		t^{-\beta}\,\partial_t^\ell (e^{t\Delta_{\H^n}}v_k)(x)\to 0
		\qquad \text{as } t\to 0^+ \text{ and as } t\to \infty,
		\]
		thanks to the factors \(\exp(-\kappa^2/(4t))\) and \(\exp(-\frac{(n-1)^2}{4}t)\), respectively.
		Therefore, all boundary terms arising in the integration-by-parts procedure vanish.
		
		Hence, integrating by parts \(m\) times in \eqref{eq_100_5}, we obtain
		\begin{equation}\label{eq_100_8}
			\sum_{k=1}^N \frac{c_k}{\Gamma(-\alpha_k)}\int_0^\infty  (e^{t\Delta_{\H^n}} v_k)(x)\, \frac{dt}{t^{m+1+\alpha_k}}=0,
			\qquad x\in\omega,\ \ m\in\N,
		\end{equation}
		where
		\begin{equation}\label{eq_100_9}
			c_k=(1+\alpha_k)(2+\alpha_k)\cdots (m+\alpha_k)
			=\frac{\Gamma(m+1+\alpha_k)}{\Gamma(1+\alpha_k)}.
		\end{equation}
		This is precisely \eqref{eq_important_m}.
		
		Fix $x\in\omega$ and define
		\begin{equation}\label{eq_100_10}
			f_k(t):=\frac{1}{\Gamma(-\alpha_k)\Gamma(1+\alpha_k)}(e^{t\Delta_{\H^n}} v_k)(x)\,t^{-(1+\alpha_k)}
			= -\frac{\sin(\pi\alpha_k)}{\pi}\,(e^{t\Delta_{\H^n}} v_k)(x)\,t^{-(1+\alpha_k)}.
		\end{equation}
		Then \eqref{eq_100_8}--\eqref{eq_100_10} imply
		\[
		\sum_{k=1}^N \Gamma(m+1+\alpha_k)\int_0^\infty f_k(t)\, t^{-m}\,dt=0,
		\qquad m\in\N.
		\]
		Moreover, from \eqref{eq_100_7} (with $\ell=0$), it is easy to see that $f_k$ satisfies the exponential decay condition \eqref{key exponential decay condition} for both $t\to 0^+$ and $t\to \infty$.
		Therefore Proposition~\ref{prop: key of entangle} yields $f_k\equiv 0$ on $(0,\infty)$, hence
		\[
		(e^{t\Delta_{\H^n}} v_k)(x)=0
		\quad \text{for all }(x,t)\in \omega\times (0,\infty),\ \ k=1,\ldots,N.
		\]
		
		For each $k$, the function $U_k(x,t):=(e^{t\Delta_{\H^n}} v_k)(x)$ solves the heat equation
		\[
		\begin{cases}
			(\partial_t-\Delta_{\H^n})U_k=0&\text{ in }\H^n\times (0,\infty),
			\\
			U_k(\cdot,0)=v_k &\text{ in }\H^n.
		\end{cases}
		\]
		By unique continuation for parabolic operators (see, e.g., \cite[Sections 1 and 4]{LinFH_UCP_para}), the vanishing of $U_k$ on the nontrivial open set $\omega\times(0,\infty)$ implies $U_k\equiv 0$ on $\H^n\times(0,\infty)$.
		Taking $t\to 0^+$ yields $v_k\equiv 0$ on $\H^n$ for all $k=1,\ldots,N$.
	\end{proof}

	\section{Application to inverse problems}\label{sec:IP}
	
	To prove Theorem~\ref{thm: main}, we combine the Runge approximation property with the integral identity from Lemma~\ref{lem:integral_identity_hyp}.
	We first record a simple observation: compact support in a bounded set implies the weak exponential decay condition used earlier.
	
	\begin{lemma}[Compact support implies an exponential decay condition]\label{lem:compact_support_decay}
		Let \(r\ge 0\) and let \(\Omega\subset \H^n\) be bounded. 
		For \(\gamma>\frac{n-1}{2}\), define
		\[
		\rho_\gamma(x):=e^{-\gamma \sqrt{1+d_{\H^n}(e_0,x)^2}}, \qquad x\in \H^n.
		\]
		Then \(\rho_\gamma\in C^\infty(\H^n)\), and there exists \(C>0\) such that
		\begin{equation}\label{eq:decay_from_compact_support}
			|\langle w,\phi\rangle|
			\le C\,\|\rho_\gamma \phi\|_{H^r(\H^n)}
			\qquad \text{for all }\phi\in C_c^\infty(\H^n)
		\end{equation}
		whenever \(w\in \widetilde H^r(\Omega)\).
	\end{lemma}
	
	\begin{proof}
		Since \(\Omega\) is bounded, we may choose \(\eta\in C_c^\infty(\H^n)\) such that
		\[
		\eta \equiv 1 \quad \text{on } \Omega.
		\]
		Because \(w\in \widetilde H^r(\Omega)\), we have \(\supp w\subset \overline{\Omega}\), and therefore
		\[
		\langle w,\phi\rangle=\langle w,\eta\phi\rangle
		\qquad \text{for all }\phi\in C_c^\infty(\H^n).
		\]
		Using the dual pairing between \(H^r(\H^n)\) and \(H^{-r}(\H^n)\), we obtain
		\[
		|\langle w,\phi\rangle|
		=|\langle w,\eta\phi\rangle|
		\le \|w\|_{H^r(\H^n)}\,\|\eta\phi\|_{H^{-r}(\H^n)}.
		\]
		
		Now set
		\[
		a_\gamma := \eta\,\rho_\gamma^{-1}.
		\]
		Since \(\eta\in C_c^\infty(\H^n)\) and \(\rho_\gamma\in C^\infty(\H^n)\) is strictly positive, we have
		\[
		a_\gamma \in C_c^\infty(\H^n).
		\]
		Moreover,
		\[
		\eta\phi = a_\gamma\,(\rho_\gamma\phi).
		\]
		Since multiplication by a fixed \(C_c^\infty\)-function is bounded on \(H^{-r}(\H^n)\), it follows that
		\[
		\|\eta\phi\|_{H^{-r}(\H^n)}
		=\|a_\gamma(\rho_\gamma\phi)\|_{H^{-r}(\H^n)}
		\le C\,\|\rho_\gamma\phi\|_{H^{-r}(\H^n)}.
		\]
		Using the continuous embedding \(H^r(\H^n)\hookrightarrow H^{-r}(\H^n)\), we further get
		\[
		\|\rho_\gamma\phi\|_{H^{-r}(\H^n)}
		\le C\,\|\rho_\gamma\phi\|_{H^r(\H^n)}.
		\]
		Combining the estimates above yields
		\[
		|\langle w,\phi\rangle|
		\le C\,\|w\|_{H^r(\H^n)}\,\|\rho_\gamma\phi\|_{H^r(\H^n)},
		\]
		which proves \eqref{eq:decay_from_compact_support}.
	\end{proof}
	
	Denote by $\mathcal P_{\H^n}$ the operator introduced in \eqref{eq: fractional poly op.}, and set
	\[
	\mathcal L_q:=\mathcal P_{\H^n}+q.
	\]
	We now establish a Runge approximation property for $\mathcal L_q$.
	
	\begin{theorem}[Runge approximation]\label{thm:Runge_hyp}
		Let $W\subset \Omega_e$ be a nonempty open subset. Then the set
		\[
		\mathcal R_W
		:=
		\left\{
		u|_\Omega \;:\;
		u\in H^{s_N}(\H^n),\
		\mathcal L_q u=0 \text{ in }\Omega,\
		u|_{\Omega_e}=f \text{ with } f\in C_c^\infty(W)
		\right\}
		\]
		is dense in $L^2(\Omega)$.
		Equivalently, for any $f\in L^2(\Omega)$ and any $\varepsilon>0$ there exists
		$u\in H^{s_N}(\H^n)$ such that
		\[
		\mathcal L_q u=0 \text{ in }\Omega,\qquad
		\supp(u|_{\Omega_e})\subset W,\qquad
		\|u-f\|_{L^2(\Omega)}<\varepsilon.
		\]
	\end{theorem}
	
	\begin{proof}
		To prove the density of \(\mathcal R_W\) in \(L^2(\Omega)\), it suffices to show that if
		\(h\in L^2(\Omega)\) satisfies
		\begin{equation}\label{eq:Runge_orth}
			(h,u)_{L^2(\Omega)}=0
			\quad \text{for all }u\in H^{s_N}(\H^n)\text{ such that }
			\mathcal L_q u=0 \text{ in }\Omega,\ \supp(u|_{\Omega_e})\subset W,
		\end{equation}
		then \(h=0\).
		
		Let \(w\in \widetilde H^{s_N}(\Omega)\) be the unique solution of
		\[
		\begin{cases}
			\mathcal L_q w=h & \text{ in }\Omega,\\
			w=0 & \text{ in }\Omega_e,
		\end{cases}
		\]
		whose existence and uniqueness follow from Lemma~\ref{lem: well-posedness} together with
		\eqref{eq: eigenvalue condition}.
		
		Now, fix \(\varphi\in C_c^\infty(W)\), and let \(u_\varphi\in H^{s_N}(\H^n)\) be the unique solution of
		\[
		\begin{cases}
			\mathcal L_q u_\varphi=0 & \text{ in }\Omega,\\
			u_\varphi=\varphi & \text{ in }\Omega_e.
		\end{cases}
		\]
		Since \(\supp(u_\varphi|_{\Omega_e})\subset W\), the orthogonality assumption \eqref{eq:Runge_orth} yields
		\begin{equation}\label{eq:Runge_uphi_orth}
			(h,u_\varphi)_{L^2(\Omega)}=0.
		\end{equation}
		
		Moreover, since \(u_\varphi-\varphi\in \widetilde H^{s_N}(\Omega)\), the weak formulation for \(w\) gives
		\[
		B_q(w,u_\varphi-\varphi)=(h,u_\varphi-\varphi)_{L^2(\Omega)}.
		\]
		Because \(\varphi\) is supported in \(W\subset \Omega_e\), we have \(\varphi|_\Omega=0\), and hence
		\[
		(h,u_\varphi-\varphi)_{L^2(\Omega)}=(h,u_\varphi)_{L^2(\Omega)}=0
		\]
		by \eqref{eq:Runge_uphi_orth}. Therefore,
		\begin{equation}\label{eq:Runge_w_uphi_minus_phi}
			B_q(w,u_\varphi-\varphi)=0.
		\end{equation}
		
		On the other hand, since \(u_\varphi\) solves \(\mathcal L_q u_\varphi=0\) in \(\Omega\) and
		\(w\in \widetilde H^{s_N}(\Omega)\), we have
		\[
		B_q(u_\varphi,w)=0.
		\]
		By the symmetry of \(B_q\), it follows that
		\begin{equation}\label{eq:Runge_w_uphi_zero}
			B_q(w,u_\varphi)=0.
		\end{equation}
		Subtracting \eqref{eq:Runge_w_uphi_minus_phi} from \eqref{eq:Runge_w_uphi_zero}, we obtain
		\[
		B_q(w,\varphi)=0.
		\]
		
		Since \(w=0\) in \(\Omega_e\) and \(\varphi\in C_c^\infty(W)\subset C_c^\infty(\Omega_e)\), the definition of
		\(B_q\) shows that
		\[
		B_q(w,\varphi)=\int_{\Omega_e} \mathcal P_{\H^n} w\,\varphi\, dV
		=\int_W \mathcal P_{\H^n} w\,\varphi\, dV.
		\]
		Hence
		\[
		\int_W \mathcal P_{\H^n} w\,\varphi\, dV =0
		\qquad \text{for all }\varphi\in C_c^\infty(W),
		\]
		which implies that
		\[
		\mathcal P_{\H^n} w=0 \qquad \text{in }W
		\]
		in the distributional sense. Since \(w=0\) in \(\Omega_e\), we also have
		\[
		w=0 \qquad \text{in }W.
		\]
		
		Furthermore, since \(w\in \widetilde H^{s_N}(\Omega)\) and \(\Omega\) is bounded, Lemma~\ref{lem:compact_support_decay}
		shows that \(w\) satisfies the weak exponential decay condition \eqref{eq: exponential decay condition weakly}.
		Therefore, we may apply the entanglement principle, Theorem~\ref{thm: ent}, to conclude that
		\[
		w\equiv 0 \qquad \text{in }\H^n.
		\]
		In particular,
		\[
		h=\mathcal L_q w=0 \qquad \text{in }\Omega.
		\]
		
		Thus the only element of \(L^2(\Omega)\) orthogonal to \(\mathcal R_W\) is \(0\), and therefore
		\(\mathcal R_W\) is dense in \(L^2(\Omega)\).
	\end{proof}
	
	\begin{proof}[Proof of Theorem~\ref{thm: main}]
		Since the DN maps agree on the measurement sets, we have
		\begin{equation}\label{eq:DN_equal_assumption}
			\left\langle (\Lambda_{q_1}-\Lambda_{q_2}) f, g\right\rangle =0
			\quad \text{for all } f\in C_c^\infty(W_1),\ g\in C_c^\infty(W_2),
		\end{equation}
		where $W_1,W_2\subset \Omega_e$ are nonempty open sets.
		
		Fix $f\in C_c^\infty(W_1)$ and $g\in C_c^\infty(W_2)$.
		Let $u_1,u_2\in H^{s_N}(\H^n)$ be the unique solutions of
		\[
		\begin{cases}
			(\mathcal P_{\H^n}+q_1)u_1=0 &\text{in }\Omega,\\
			u_1=f &\text{in }\Omega_e,
		\end{cases}
		\qquad
		\begin{cases}
			(\mathcal P_{\H^n}+q_2)u_2=0 &\text{in }\Omega,\\
			u_2=g &\text{in }\Omega_e.
		\end{cases}
		\]
		By Lemma~\ref{lem:DNmap_hyp}, Lemma~\ref{lem:integral_identity_hyp} and \eqref{eq:DN_equal_assumption},
		\begin{equation}\label{eq:int_id_zero}
			\int_\Omega (q_1-q_2)u_1u_2 \, dV=0
			\quad \text{for all } f\in C_c^\infty(W_1),\ g\in C_c^\infty(W_2).
		\end{equation}
		
		Let $\varphi\in L^2(\Omega)$ be arbitrary. By the Runge approximation property (Theorem~\ref{thm:Runge_hyp}),
		there exist sequences $\{u_1^{(k)}\}_k,\{u_2^{(k)}\}_k\subset H^{s_N}(\H^n)$ such that
		\[
		(\mathcal P_{\H^n}+q_1)u_1^{(k)}=0 \text{ in }\Omega,\quad
		\supp(u_1^{(k)}|_{\Omega_e})\subset W_1,\quad
		u_1^{(k)}|_\Omega \to \varphi \text{ in }L^2(\Omega),
		\]
		and
		\[
		(\mathcal P_{\H^n}+q_2)u_2^{(k)}=0 \text{ in }\Omega,\quad
		\supp(u_2^{(k)}|_{\Omega_e})\subset W_2,\quad
		u_2^{(k)}|_\Omega \to 1 \text{ in }L^2(\Omega).
		\]
		Applying \eqref{eq:int_id_zero} with $u_1=u_1^{(k)}$ and $u_2=u_2^{(k)}$ gives
		\[
		\int_\Omega (q_1-q_2)u_1^{(k)}u_2^{(k)} \, dV=0
		\quad \text{for all }k.
		\]
		Since $q_1-q_2\in L^\infty(\Omega)$ and $u_1^{(k)}\to \varphi$, $u_2^{(k)}\to 1$ in $L^2(\Omega)$,
		we have $u_1^{(k)}u_2^{(k)}\to \varphi$ in $L^1(\Omega)$ (by Cauchy--Schwarz). Letting $k\to\infty$ yields
		\[
		\int_\Omega (q_1-q_2)\varphi\, dV=0
		\qquad \text{for all }\varphi\in L^2(\Omega).
		\]
		Thanks to the arbitrariness of $\varphi \in L^2(\Omega)$, we conclude that $q_1=q_2$ almost everywhere in $\Omega$.
	\end{proof}

	\section*{Statements and Declarations}
	
	\noindent\textbf{Data availability statement.} 
	No datasets were generated or analyzed during the current study.
	
	\medskip 
	
	\noindent\textbf{Conflict of Interest.}  The author declares that there are no conflicts of interest.

	\medskip

	\medskip

	\noindent\textbf{Acknowledgments.} 
	Y.-H. L. is partially supported by the National Science and Technology Council (NSTC) of Taiwan, under the project 113-2628-M-A49-003. Y.-H. L. acknowledges financial support from the Alexander von Humboldt Foundation through the Henriette Herz Scouting Programme and hosted by Universität Duisburg-Essen.

	\bibliography{refs} 

\begin{thebibliography}{CMRU22}

\bibitem[BGS15]{BGS_15}
Valeria Banica, Mar\'ia del~Mar Gonz\'alez, and Mariel S\'aez.
\newblock Some constructions for the fractional {L}aplacian on noncompact
  manifolds.
\newblock {\em Rev. Mat. Iberoam.}, 31(2):681--712, 2015.

\bibitem[CGRU23]{CGRU2023reduction}
Giovanni Covi, Tuhin Ghosh, Angkana R{\"u}land, and Gunther Uhlmann.
\newblock A reduction of the fractional {C}alder\'on problem to the local
  {C}alder\'on problem by means of the {C}affarelli-{S}ilvestre extension.
\newblock {\em arXiv preprint arXiv:2305.04227}, 2023.

\bibitem[CLL19]{CLL2017simultaneously}
Xinlin Cao, Yi-Hsuan Lin, and Hongyu Liu.
\newblock Simultaneously recovering potentials and embedded obstacles for
  anisotropic fractional {S}chr\"{o}dinger operators.
\newblock {\em Inverse Probl. Imaging}, 13(1):197--210, 2019.

\bibitem[CLR20]{cekic2020calderon}
Mihajlo Ceki\'c, Yi-Hsuan Lin, and Angkana R\"uland.
\newblock The {C}alder\'on problem for the fractional {S}chr\"odinger equation
  with drift.
\newblock {\em Calc. Var. Partial Differential Equations}, 59(3):Paper No. 91,
  46, 2020.

\bibitem[CMRU22]{CMRU20}
Giovanni Covi, Keijo M\"{o}nkk\"{o}nen, Jesse Railo, and Gunther Uhlmann.
\newblock The higher order fractional {C}alder\'{o}n problem for linear local
  operators: {U}niqueness.
\newblock {\em Adv. Math.}, 399:Paper No. 108246, 2022.

\bibitem[CS07]{CS07}
Luis Caffarelli and Luis Silvestre.
\newblock An extension problem related to the fractional {L}aplacian.
\newblock {\em Comm. Partial Differential Equations}, 32(7-9):1245--1260, 2007.

\bibitem[DM88]{DM88_heatkernel_hyper}
Edward~Brian Davies and Nikolaos Mandouvalos.
\newblock Heat kernel bounds on hyperbolic space and {K}leinian groups.
\newblock {\em Proc. London Math. Soc. (3)}, 57(1):182--208, 1988.

\bibitem[Fei24]{Fei24_TAMS}
Ali Feizmohammadi.
\newblock Fractional {C}alder\'on problem on a closed {R}iemannian manifold.
\newblock {\em Trans. Amer. Math. Soc.}, 377(4):2991--3013, 2024.

\bibitem[FGKU21]{feizmohammadi2021fractional}
Ali Feizmohammadi, Tuhin Ghosh, Katya Krupchyk, and Gunther Uhlmann.
\newblock Fractional anisotropic {C}alder\'on problem on closed {R}iemannian
  manifolds.
\newblock {\em arXiv:2112.03480}, 2021.

\bibitem[FGKU25]{FGKU_2025fractional}
Ali Feizmohammadi, Tuhin Ghosh, Katya Krupchyk, and Gunther Uhlmann.
\newblock Fractional anisotropic {C}alder\'on problem on closed {R}iemannian
  manifolds.
\newblock {\em J. Differential Geom.}, 131(2):401--414, 2025.

\bibitem[FKU24]{FKU24}
Ali Feizmohammadi, Katya Krupchyk, and Gunther Uhlmann.
\newblock Calder\'{o}n problem for fractional {S}chr\"{o}dinger operators on
  closed {R}iemannian manifolds.
\newblock {\em arXiv preprint arXiv:2407.16866}, 2024.

\bibitem[FL24]{FL24}
Ali Feizmohammadi and Yi-Hsuan Lin.
\newblock Entanglement principle for the fractional {L}aplacian with
  applications to inverse problems.
\newblock {\em arXiv preprint arXiv:2412.13118}, 2024.

\bibitem[GGG03]{GGG_book_03}
Izrail~Moiseevich Gelfand, Simon~G. Gindikin, and Mark~Iosifovich Graev.
\newblock {\em Selected topics in integral geometry}, volume 220 of {\em
  Translations of Mathematical Monographs}.
\newblock American Mathematical Society, Providence, RI, 2003.
\newblock Translated from the 2000 Russian original by A. Shtern.

\bibitem[GLX17]{GLX}
Tuhin Ghosh, Yi-Hsuan Lin, and Jingni Xiao.
\newblock The {C}alder\'{o}n problem for variable coefficients nonlocal
  elliptic operators.
\newblock {\em Comm. Partial Differential Equations}, 42(12):1923--1961, 2017.

\bibitem[GRSU20]{GRSU20}
Tuhin Ghosh, Angkana R\"{u}land, Mikko Salo, and Gunther Uhlmann.
\newblock Uniqueness and reconstruction for the fractional {C}alder\'{o}n
  problem with a single measurement.
\newblock {\em J. Funct. Anal.}, 279(1):108505, 42, 2020.

\bibitem[GSU20]{GSU20}
Tuhin Ghosh, Mikko Salo, and Gunther Uhlmann.
\newblock The {C}alder\'{o}n problem for the fractional {S}chr\"{o}dinger
  equation.
\newblock {\em Anal. PDE}, 13(2):455--475, 2020.

\bibitem[GU21]{GU2021calder}
Tuhin Ghosh and Gunther Uhlmann.
\newblock The {C}alder\'{o}n problem for nonlocal operators.
\newblock {\em arXiv:2110.09265}, 2021.

\bibitem[Hel08]{Hel_book_08}
Sigurdur Helgason.
\newblock {\em Geometric analysis on symmetric spaces}, volume~39 of {\em
  Mathematical Surveys and Monographs}.
\newblock American Mathematical Society, Providence, RI, second edition, 2008.

\bibitem[HL19]{harrach2017nonlocal-monotonicity}
Bastian Harrach and Yi-Hsuan Lin.
\newblock Monotonicity-based inversion of the fractional {S}chr\"{o}dinger
  equation {I}. {P}ositive potentials.
\newblock {\em SIAM J. Math. Anal.}, 51(4):3092--3111, 2019.

\bibitem[HL20]{harrach2020monotonicity}
Bastian Harrach and Yi-Hsuan Lin.
\newblock Monotonicity-based inversion of the fractional {S}ch\"{o}dinger
  equation {II}. {G}eneral potentials and stability.
\newblock {\em SIAM J. Math. Anal.}, 52(1):402--436, 2020.

\bibitem[KLW22]{KLW2021calder}
Pu-Zhao Kow, Yi-Hsuan Lin, and Jenn-Nan Wang.
\newblock The {C}alder\'{o}n problem for the fractional wave equation:
  uniqueness and optimal stability.
\newblock {\em SIAM J. Math. Anal.}, 54(3):3379--3419, 2022.

\bibitem[Lin90]{LinFH_UCP_para}
Fang-Hua Lin.
\newblock A uniqueness theorem for parabolic equations.
\newblock {\em Comm. Pure Appl. Math.}, 43(1):127--136, 1990.

\bibitem[Lin22]{lin2020monotonicity}
Yi-Hsuan Lin.
\newblock Monotonicity-based inversion of fractional semilinear elliptic
  equations with power type nonlinearities.
\newblock {\em Calc. Var. Partial Differential Equations}, 61(5):Paper No. 188,
  30, 2022.

\bibitem[Lin26]{lin2024fractional}
Yi-Hsuan Lin.
\newblock The fractional anisotropic {C}alder\'on problem for a nonlocal
  parabolic equation on closed {R}iemannian manifolds.
\newblock {\em Commun. Anal. Comput.}, 7:60--75, 2026.

\bibitem[LL22]{LL2020inverse}
Ru-Yu Lai and Yi-Hsuan Lin.
\newblock Inverse problems for fractional semilinear elliptic equations.
\newblock {\em Nonlinear Anal.}, 216:Paper No. 112699, 21, 2022.

\bibitem[LL23]{LL2022inverse}
Yi-Hsuan Lin and Hongyu Liu.
\newblock Inverse problems for fractional equations with a minimal number of
  measurements.
\newblock {\em Commun. Anal. Comput.}, 1(1):72--93, 2023.

\bibitem[LL25]{LL25_Integro}
Yi-Hsuan Lin and Hongyu Liu.
\newblock {\em Inverse {P}roblems for {I}ntegro-differential {O}perators},
  volume 222 of {\em Applied Mathematical Sciences}.
\newblock Springer, Cham, 2025.

\bibitem[LLR20]{LLR2019calder}
Ru-Yu Lai, Yi-Hsuan Lin, and Angkana R\"{u}land.
\newblock The {C}alder\'{o}n problem for a space-time fractional parabolic
  equation.
\newblock {\em SIAM J. Math. Anal.}, 52(3):2655--2688, 2020.

\bibitem[LLU22]{LLU2022para}
Ching-Lung Lin, Yi-Hsuan Lin, and Gunther Uhlmann.
\newblock The {C}alder\'{o}n problem for nonlocal parabolic operators.
\newblock {\em arXiv preprint arXiv:2209.11157}, 2022.

\bibitem[LLU23]{LLU2023calder}
Ching-Lung Lin, Yi-Hsuan Lin, and Gunther Uhlmann.
\newblock The {C}alder\'{o}n problem for nonlocal parabolic operators: {A} new
  reduction from the nonlocal to the local.
\newblock {\em arXiv preprint arXiv:2308.09654}, 2023.

\bibitem[LLY25]{LLY25_entangle}
Ru-Yu Lai, Yi-Hsuan Lin, and Lili Yan.
\newblock Entanglement principle and fractional {C}alder\'on problem for
  nonlocal parabolic operators.
\newblock {\em arXiv preprint arXiv:2510.18641}, 2025.

\bibitem[LZ23]{LZ2023unique}
Yi-Hsuan Lin and Philipp Zimmermann.
\newblock Unique determination of coefficients and kernel in nonlocal porous
  medium equations with absorption term.
\newblock {\em arXiv preprint arXiv:2305.16282}, 2023.

\bibitem[LZ24]{LZ2024approximation}
Yi-Hsuan Lin and Philipp Zimmermann.
\newblock Approximation and uniqueness results for the nonlocal diffuse optical
  tomography problem.
\newblock {\em arXiv preprint arXiv:2406.06226}, 2024.

\bibitem[McL00]{ML-strongly-elliptic-systems}
William McLean.
\newblock {\em Strongly elliptic systems and boundary integral equations}.
\newblock Cambridge University Press, Cambridge, 2000.

\bibitem[RS18]{ruland2018exponential}
Angkana R\"{u}land and Mikko Salo.
\newblock Exponential instability in the fractional {C}alder\'{o}n problem.
\newblock {\em Inverse Problems}, 34(4):045003, 21, 2018.

\bibitem[RS20]{RS20}
Angkana R\"{u}land and Mikko Salo.
\newblock The fractional {C}alder\'{o}n problem: low regularity and stability.
\newblock {\em Nonlinear Anal.}, 193:111529, 56, 2020.

\bibitem[R{\"u}l15]{ruland2015unique}
Angkana R{\"u}land.
\newblock Unique continuation for fractional {S}chr\"odinger equations with
  rough potentials.
\newblock {\em Comm. Partial Differential Equations}, 40(1):77--114, 2015.

\bibitem[R{\"u}l23]{ruland2023revisiting}
Angkana R{\"u}land.
\newblock Revisiting the anisotropic fractional {C}alder\'on problem using the
  {C}affarelli-{S}ilvestre extension.
\newblock {\em arXiv preprint arXiv:2309.00858}, 2023.

\bibitem[Tat01]{Tataru_strichartz}
Daniel Tataru.
\newblock Strichartz estimates in the hyperbolic space and global existence for
  the semilinear wave equation.
\newblock {\em Trans. Amer. Math. Soc.}, 353(2):795--807, 2001.

\end{thebibliography}
	
	\bibliographystyle{alpha}

\end{document}